\documentclass[letter,10pt]{amsart}

\usepackage{amsmath,amssymb,amsfonts,amsthm,ascmac, amscd, eucal,setspace}

\usepackage{float}

\usepackage{caption}
\usepackage{subcaption}
\usepackage{hyperref}

\usepackage{url}
\usepackage{latexsym}
\def\qed{\hfill$\Box$}   
\usepackage{enumerate}

\usepackage{tikz-cd}
\usepackage{tikz}
\usetikzlibrary{snakes, 
	3d, matrix,decorations.pathreplacing,calc,decorations.pathmorphing}
\usetikzlibrary{decorations.markings}
\usetikzlibrary {positioning}

\numberwithin{equation}{section}
\allowdisplaybreaks[1]

\renewcommand\thesubfigure{(\arabic{subfigure})}
\newcommand{\defi}[1]{{\textit{#1}}}
\newcommand{\flag}{{\mathcal{F} \ell}}

\def\C{\mathbb C}
\def\R{\mathbb R}
\def\Q{\mathbb Q}
\def\Z{\mathbb Z}
\def\Ref{E}
\def\Inv{\mathrm{Inv}}
\def\uw{u_w}
\def\Cstar{\mathbb C^{\ast}}
\def\i{\underline{i}}
\def\j{\underline{j}}

\def\al{{\alpha}}
\def\be{{\beta}}
\DeclareMathOperator{\Hom}{Hom}
\DeclareMathOperator{\Int}{Int}
\DeclareMathOperator{\GL}{GL}
\def\a{\mathbf a}
\def\G{\Gamma}
\def\Xw{X_w}
\def\Yw{Y_w}
\def\Ywo{Y_{w_0}}
\def\dwu{d_w(u)}
\def\awu{a_w(u)}
\def\dwuo{d_{w_0}(u)}
\def\awuo{a_{w_0}(u)}
\def\Refw{\Ref_w}
\def\Refwu{\Ref_w(u)}
\def\wideRefwu{\widetilde{\Ref}_w(u)}
\def\Refwou{\Ref_{w_0}(u)}
\def\Gwu{\G_w(u)}
\def\Gw{\G_w}
\def\Dw{D_w}
\def\Dwu{D_w(u)}
\def\Dwou{D_{w_0}(u)}
\def\Cv{C(v)}
\def\Cu{C(u)}
\def\Cbarv{C(\bar{v})}
\def\Cwx{C_w(x)}
\def\Cwy{C_w(y)}
\def\Cvsi{C(vs_i)}

\def\Cbw{C_w}
\def\Cbwu{C_w(u)}
\def\Iw{I_w}
\def\Iwo{I_{w_0}}
\def\Jw{J_{w}}
\def\MO{\mathcal{O}}
\def\oMO{\overline{\MO}}
\def\UZ{\Z^n}
\def\U{\R^n}
\def\V{\R^n}

\newcommand{\Rw}{{R_w}}

\theoremstyle{plain}
\newtheorem{theorem}{Theorem}[section]
\newtheorem{lemma}[theorem]{Lemma}
\newtheorem{proposition}[theorem]{Proposition}
\newtheorem{corollary}[theorem]{Corollary}

\newtheorem{conjecture}[theorem]{Conjecture}

\theoremstyle{definition}
\newtheorem{example}[theorem]{Example}
\newtheorem{definition}[theorem]{Definition}
\newtheorem{remark}[theorem]{Remark}

\begin{document}
	\tikzset{->-/.style={decoration={
				markings,
				mark=at position #1 with {\arrow{>}}},postaction={decorate}}}

	\author{Eunjeong Lee}
	\address{Center for Geometry and Physics, Institute for Basic Science (IBS), Pohang 37673, Korea}
	\email{eunjeong.lee@ibs.re.kr}
	
	\author{Mikiya Masuda}
	\address{Department of Mathematics, Graduate School of Science, Osaka City University, Sumiyoshi-ku, Sugimoto, 558-8585, Osaka, Japan}
	\email{masuda@sci.osaka-cu.ac.jp}
	
	\thanks{Lee was partially supported by Basic Science 
		Research Program through the National Research Foundation of Korea (NRF) 
		funded by the Ministry of  Science, ICT \& Future Planning 
		(No. 2016R1A2B4010823), and IBS-R003-D1. Masuda was partially supported by JSPS Grant-in-Aid for Scientific Research 16K05152 and the bilateral program \lq\lq Topology and geometry of torus actions, cohomological rigidity, and hyperbolic manifolds\rq\rq~between JSPS and RFBR}
	


\title[Generic torus orbit closures in Schubert varieties]{Generic torus orbit closures \\
	in Schubert varieties}
\date{\today}
\maketitle
	
\begin{abstract}
The closure of a generic torus orbit in the flag variety $G/B$ of type $A_{n-1}$ is known to be a permutohedral variety and well studied.  In this paper we introduce the notion of a generic torus orbit in the Schubert variety $\Xw$ $(w\in \mathfrak{S}_n)$ and study its closure $\Yw$.  We identify the maximal cone in the fan of $\Yw$ corresponding to a fixed point $uB$ $(u\le w)$, associate a graph $\Gwu$ to each $u\le w$, and show that $\Yw$ is smooth at $uB$ if and only if $\Gwu$ is a forest.  We also introduce a polynomial $A_w(t)$ for each $w$, which agrees with the Eulerian polynomial when $w$ is the longest element of $\mathfrak{S}_n$, and show that the Poincar\'e polynomial of $\Yw$ agrees with $A_w(t^2)$ when $\Yw$ is smooth.   
\end{abstract}
	
%
%


\section{Introduction} \label{sect:1}
Let $G$ be the general linear group $\GL_{n}(\C)$, let $B \subset G$ be the Borel subgroup of upper triangular matrices, and~let $T \subset B$ be the torus subgroup of diagonal matrices. The left multiplication by~$T$ on $G$ induces the $T$-action on the flag variety $G/B$. The set of $T$-fixed points in~$G/B$ bijectively corresponds to the symmetric group $\mathfrak{S}_n$ on the set $\{1,2,\dots,n\}$. 

Let $\MO$ be a $T$-orbit of an element of $G/B$ and $\oMO$ its closure. It is known that $\oMO$ is normal (\cite[Proposition~4.8]{CK00}), so $\oMO$ is a toric variety. When $\MO$ is {\em generic}, which means that the closure $\oMO$ contains all the $T$-fixed points in $G/B$, $\oMO$ is known to be the permutohedral variety of complex dimension $n-1$. The maximal cones in the fan of the permutohedral variety are the Weyl chambers in type $A_{n-1}$. The permutohedral variety appears in many areas in mathematics and is well studied (see, for instance, \cite{Abe15Young},  \cite{DeMariShayman88}, \cite{Huh14Rota}, \cite{Klyachko85Orbits}, \cite{Postnikov09_Permutohedra} and \cite{Procesi90Toric}). It is smooth and its Poincar\'e polynomial is given by $A_n(t^2)$ where $A_n(t)$ denotes the Eulerian polynomial associated to $\mathfrak{S}_n$. 
In general, when $G$ is a semisimple Lie group  with the Weyl group $W$, one can also consider the toric variety $X(W)$ associated to the Weyl chambers. The corresponding polytope is called $W$-permutohedron in~\cite{FR07_root} and the representation of Weyl group $W$ on the cohomology of $X(W)$ has been studied in~ \cite{DL94character}, \cite{Procesi90Toric}, and \cite{Stembridge94permutation}.

For an element $w \in \mathfrak{S}_{n}$ we denote by $\Xw$ the \defi{Schubert variety} $\overline{BwB/B}$ in $G/B$. The $T$-action on $G/B$ leaves $\Xw$ invariant. We say that a $T$-orbit $\MO$ in $\Xw$ is {\em generic in $\Xw$} if the closure $\oMO$ contains all the $T$-fixed points in $\Xw$. Here the~$T$-fixed points in $\Xw$ are $uB$'s for $u\leq w$ in Bruhat order. When $w$ is the longest element $w_0$ of $\mathfrak{S}_n$, our notion of {\em generic in $\Xw$} agrees with the {\em generic} mentioned in the previous paragraph but it is otherwise different. 

In this paper, we study the closure, denoted by $\Yw$, of a generic $T$-orbit in $\Xw$. 
First we identify maximal cones in the fan of $\Yw$. To each $v\in \mathfrak{S}_n$ we associate a cone in $\R^n$: 
\[
C(v):=\{ (a_1,\dots,a_n)\in \R^n\mid a_{v(1)}\le a_{v(2)}\le \cdots \le a_{v(n)}\}.
\]
These cones projected on the quotient vector space $\R^n/\R(1,\dots,1)$ are exactly the maximal cones in the fan of the permutohedral variety $\Ywo$. To describe the maximal cones in the fan of $\Yw$, 
we introduce an operation $\Rw \colon \mathfrak{S}_n \to \mathfrak{S}_n$, called \defi{a retraction}, which satisfies the following properties:  
\[
\Rw(v)\le w \text{ for any $v$}, \quad \Rw(v)=v \text{ if $v\le w$},\quad \text{and}\quad \Rw(w_0)=w.
\]
(See Section~\ref{sect:3} for the definition of the retraction $\Rw$.) With this understood we have:

\begin{theorem}[Corollary~\ref{coro:maximal_cone}] \label{theo:main1}
	The maximal cone in the fan of $\Yw$ corresponding to a fixed point $uB$ in $\Yw$ is the union of the $C(v)$'s with $\Rw(v)=u$ projected on the quotient vector space $\R^n/F_w$, where the linear subspace $F_w$ of $\R^n$ is determined by the subtorus of $T$ which fixes $\Yw$ pointwise {\rm (}see Remark~\ref{rema:ambient_space} for the definition of $F_w${\rm )}. 
\end{theorem} 

Using this description of the maximal cones (in fact, considering their dual cones), we obtain a criterion of smoothness of $\Yw$ in terms of graphs. Indeed, we associate a graph $\Gwu$ to each $u\le w$ and prove the following. 

\begin{theorem}[Corollary~\ref{cor_Duw_smooth_iff_Guw_forest}] \label{theo:main2}
	The generic torus orbit closure $\Yw$ in the Schubert variety $\Xw$ is smooth at a fixed point $uB$ in $\Yw$ if and only if the graph $\Gwu$ is a forest. Therefore, $\Yw$ is smooth if and only if $\Gwu$ is a forest for every $u\le w$. 
\end{theorem}

\noindent 
To our surprise, the graph $\Gwu$ with $u=w$, abbreviated as $\Gw$, has been studied in \cite{bm-bu07} and the following three conditions are shown to be equivalent by the results~\cite[Theorem~1.1]{bm-bu07} and~\cite[Proposition~2]{wo-yo06}:
\begin{enumerate}
	\item $\Gw$ is a forest.
	\item $w$ avoids the patterns $4231$ and $45{\bar 3}12$ (see Section~\ref{sect:5}).
	\item $\Xw$ is factorial.
\end{enumerate} 
Looking at examples of $\Gwu$, we conjecture that if $\Gw$ is a forest, then so is $\Gwu$ for any $u\le w$, in other words, $\Yw$ is smooth if $\Yw$ is smooth at $wB$ (Conjecture~\ref{conj_smoothness_of_Yw}). 
One can also see that the fixed point $id B$ in $Y_w$, where $id$ denotes the identity permutation in $\mathfrak{S}_n$, is smooth in $Y_w$ for any $w\in \mathfrak{S}_n$ (Corollary~\ref{cor_Duw_smooth_iff_Guw_forest}).  These are in sharp contrast with the Schubert variety $X_w$ because the fixed point $wB$ is smooth in $X_w$ for any $w\in  \mathfrak{S}_n$ and $X_w$ is smooth if and only if $id B$ is smooth in $X_w$ (see \cite[p. 208]{BL20Singular}).      
This sharp contrast shows that if $X_w$ is a toric variety, i.e., $X_w=Y_w$, then $X_w$ is smooth because it is smooth at $id B$.  Moreover, the smoothness at $id B$ implies the result in \cite{Karu13Schubert} that $X_w$ is a toric variety if and only if $w$ is a product of distinct simple reflections (Corollary~\ref{coro:toric_Schubert}).  

There is a moment map $\mu\colon G/B\to \R^n$ whose image is the convex hull of points $(u^{-1}(1),\dots,u^{-1}(n))$ in $\R^n$ for all $u\in \mathfrak{S}_n$, that is, the permutohedron of dimension $n-1$. It follows from \cite[Theorem 2]{Atiyah82} that $\mu(\Yw)$ is the convex hull of points $(u^{-1}(1),\dots,u^{-1}(n))$ in $\R^n$ for all $u\le w$, which is the Bruhat interval polytope $Q_{id,w^{-1}}$ in \cite{ts-wi15}. 
Note that the normal fan of the Bruhat interval polytope $Q_{id, w^{-1}}$ agrees with the fan of  $Y_w$.
One can see that the graph $\Gwu$ is a forest for every $u\le w$ if and only if the polytope $Q_{id,w^{-1}}$ is simple.   
Our conjecture is equivalent to saying that the polytope $Q_{id,w^{-1}}$ is simple if $\mu(wB)=(w^{-1}(1),\dots,w^{-1}(n))$ is a simple vertex.   

We also study the Poincar\'e polynomial of $\Yw$. We introduce a polynomial $A_w(t)$ for each $w$ in a purely combinatorial way. The polynomial $A_w(t)$ agrees with the Eulerian polynomial $A_n(t)$ when $w=w_0$, so the following theorem generalizes a known result on the Poincar\'e polynomial of the permutohedral variety $\Ywo$. 

\begin{theorem}[Theorem~\ref{theo:Poincare}] \label{theo:main3} 
	If $\Yw$ is smooth, then the Poincar\'e polynomial of $\Yw$ agrees with $A_w(t^2)$ and hence the polynomial $A_w(t)$ is palindromic and unimodal.
\end{theorem}

\noindent
When $w=4231$ or $3412$, $Y_w$ is singular and $A_w(t)$ is not palindromic but $A_w(t^2)$ still agrees with the Poincar\'e polynomial of $Y_w$, so it would be interesting to ask whether $A_w(t^2)$ agrees with the Poincar\'e polynomial of $Y_w$ for any $w\in \mathfrak{S}_n$ and  palindromicity of $A_w(t)$ implies smoothness of $\Yw$. 

This paper is organized as follows. In Section~\ref{sect:2}, we discuss generic orbits and generic points in the Schubert variety $\Xw$. In Section~\ref{sect:3} we introduce the operation on $\mathfrak{S}_n$ and prove Theorem~\ref{theo:main1}.  In Section~\ref{sec_right_weak_order} we show that the set $\{v\in \mathfrak{S}_n\mid \Rw(v)=u\}$ for $u\le w$ is an interval of the right weak Bruhat order.  In Section~\ref{sect:4} we identify the cone dual to the maximal cone in the fan of $Y_w$ corresponding to $wB$. In Section~\ref{sect:5} we associate the graph $\Gw$ to the dual cone and discuss simpliciality of the dual cone. In addition, we prove that simpliciality and smoothness are the same in our case.  In Section~\ref{sect:6} we discuss the smoothness of $Y_w$ at the other fixed points $uB$ $(u<w)$. Indeed, we identify the cone dual to the maximal cone in the fan of~$Y_w$ corresponding to $uB$, introduce the graph $\Gwu$, and prove Theorem~\ref{theo:main2}. We also discuss when $X_w$ is a toric variety.  In Section~\ref{sect:7} we introduce the polynomial $A_w(t)$ and prove Theorem~\ref{theo:main3}.  In \ref{section_acyclicity_and_avoidance} we give an alternative proof to the pattern avoidance criterion of when $\Gw$ is a forest. 
In \ref{section_retraction_sequence_poincare_polynomial} we compute Poincar\'{e} polynomials of $Y_w$ when $w = 4231$ and $3412$ using retraction sequences of polytopes.

\medskip
\noindent
{\bf Acknowledgment.}  
We would like to thank the referees who read the paper carefully and gave valuable suggestions and comments.
We thank Seonjeong Park for bringing the papers \cite{bm-bu07} and \cite{ts-wi15} to our attention.  We also thank Hiraku Abe, Hiroaki Ishida, Tatsuya Horiguchi, Svjetlana Terzi\'{c}, Anatol Kirillov and Jongbaek Song for their interest in our work and helpful conversations, Jim Carrell for his valuable comments, and Masashi Noji for his computer program to find the graph $\Gamma_w(u)$.  Lee thanks Professor Dong Youp Suh
for his support throughout the project.

\section{Generic points in Schubert varieties} \label{sect:2}
Let $G = \GL_{n}(\C)$, $B \subset G$ the Borel subgroup of upper triangular matrices and $T \subset B$ the torus subgroup of diagonal matrices. Let $\mathfrak{S}_n$ be the symmetric group on the set $\{1,2,\dots,n\}$.  An element $w$ of $\mathfrak{S}_n$ defines a permutation matrix $[e_{w(1)} \ \cdots \ e_{w(n)}]$ where $e_1,\dots,e_n$ denote the standard column vectors in $\R^n$.  Through this correspondence, we think of an element of $\mathfrak{S}_n$ as an element of $G$.
For an element $w \in \mathfrak{S}_{n}$ we denote the \defi{Schubert variety} $\overline{BwB/B}$ in the flag variety $G/B$ by $\Xw$. {The left multiplication by $T$ on $G$ induces a $T$-action on $G/B$ which leaves $\Xw$ invariant. The set of $T$-fixed points in $G/B$ bijectively corresponds to the symmetric group $\mathfrak{S}_{n}$ through the map $v\in \mathfrak{S}_n\to vB\in G/B$ and $vB$ lies in $\Xw$ if and only if $v\leq w$ in Bruhat order (see \cite[\S 10.5]{Fulton97Young}).}

\begin{definition}
	{We call a $T$-orbit in $\Xw$ \textit{generic} if its closure contains all the $T$-fixed points in $\Xw$ and call a point in $\Xw$ \textit{generic} if it is in a generic $T$-orbit. We will denote the closure of a generic $T$-orbit in $\Xw$ by $\Yw$.} 
\end{definition}

We will see in Proposition~\ref{prop:genric_point} that generic points in $X_w$ form a Zariski open subset of $X_w$.
We describe these generic points in the Schubert variety $\Xw$ using the Pl\"{u}cker coordinates. 
To introduce the Pl\"{u}cker coordinate, we define the set
\[
I_{d,n} = \{ \underline{i} = (i_1,\dots,i_d) \in \Z^{d} 
\mid 1 \leq i_1 < \cdots < i_d \leq n\}.
\]
For an element $x = (x_{ij}) \in G=\GL_n(\C)$, the $\underline{i}$th Pl\"{u}cker coordinate $p_{\underline{i}}(x)$ of $x$ is given by the $d \times d$ minor of $x$, with row indices $i_1,\dots,i_d$ and the column indices $1,\dots,d$ for $\underline{i} =(i_1,\dots,i_d)  \in I_{d,n}$. 
Then the Pl\"{u}cker embedding $\psi$ is defined to be
\begin{equation}\label{eq_Plucker_embedding}
\psi \colon G/B \to \prod_{d=1}^{n-1} \mathbb{C}P^{{n\choose d}-1},
\quad xB \mapsto \prod_{d=1}^{n-1} (p_{\underline{i}}(x))_{\underline{i} \in I_{d,n}}.
\end{equation}
One can see that the Pl\"{u}cker embedding $\psi$ is well-defined.
The map $\psi$ is $T$-equivariant with respect to the action of 
$T$ on $\prod_{d=1}^{n-1} \mathbb{C} P^{{n\choose d}-1}$ given by
\[
(t_1,\dots,t_n) \cdot (p_{\i})_{\i \in I_{d,n}}
:= (t_{i_1}\cdots t_{i_d} \cdot p_{\i})_{\i \in I_{d,n}}
\]
for $(t_1,\dots,t_n) \in T$ and $\i = (i_1,\dots,i_d)$.

\begin{example}\label{example_Plucker_GL3}
	Suppose that $G = \GL_3(\mathbb{C})$. Then the Pl\"{u}cker embedding 
	$\psi \colon {G/B} \to \C {P}^{{3\choose 1}-1} \times \C {P}^{{3\choose 2}-1}$ maps an element $xB \in G/B$	to 
	\begin{align*}
	&([p_1(x),p_2(x),p_3(x)], [p_{1,2}(x), p_{1,3}(x), p_{2,3}(x)]) \\
	&\qquad = ([x_{11}, x_{21}, x_{31}], [x_{11} x_{22} - x_{21}x_{12}, x_{11}x_{32} - x_{31} x_{12}, x_{21}x_{32} - x_{31}x_{22}]).
	\end{align*}
	Here, $x = (x_{ij}) \in \GL_3(\C)$.
	Since the action of $T$ on $\GL_3(\C)$ is given by
	\begin{align*}
	(t_1,t_2,t_3) \cdot \begin{pmatrix}
	x_{11} & x_{12} & x_{13} \\
	x_{21} & x_{22} & x_{23} \\
	x_{31} & x_{32} & x_{33}
	\end{pmatrix}
	= \begin{pmatrix}
	t_1 x_{11} & t_1 x_{12} & t_1 x_{13} \\
	t_2 x_{21} & t_2 x_{22} & t_2 x_{23} \\
	t_3 x_{31} & t_3 x_{32} & t_3 x_{33}
	\end{pmatrix},
	\end{align*}
	{one can easily check that the map $\psi$ is $T$-equivariant.}
\end{example}

Given $\underline{i} = (i_1,\dots,i_d), \underline{j} = (j_1,\dots,j_d)
\in I_{d,n}$, define a partial order $\geq$ on $I_{d,n}$ by
\begin{equation}\label{equation_partial_order}
\underline{i} \geq \underline{j} 
\iff i_t \geq j_t \text{ for all } 1 \leq t \leq d.
\end{equation}
With this partial order, it is known from \cite[Theorem 3.2.10]{BL20Singular} that the ideal sheaf of the Schubert variety $\Xw$ is generated by $\{p_{\underline{i}} \mid \underline{i}\in \Iw\}$ where
\[
\Iw := \bigcup_{1 \leq d \leq n-1} \{ \underline{i} \in I_{d,n}\mid w^{(d)} \ngeq \underline{i}\}.
\]
Here $w^{(d)}$ denotes the ordered $d$-tuple obtained from $\{w(1),\dots,w(d)\}$ by arranging its elements in ascending order. It follows that 
\begin{equation} \label{eq:\Iw}
{\underline{i}} \in \Iw \implies x_{\underline{i}} := p_{\underline{i}}(x) = 0 \quad\text{for any $xB\in \Xw$}.
\end{equation}

Now we introduce
\begin{equation} \label{eq:\Jw}
\Jw := 
\bigcup_{1 \leq d \leq n-1}
\{ \underline{j} \in I_{d,n} \mid \underline{j} = v^{(d)} \text{ for some } v \leq w \}.
\end{equation}
Then $p_{\j}$ for $\j\in \Jw$ is not identically zero on $\Xw$ because if $\j=v^{(d)}$ for some $v\le w$, then $p_{\j}(v)=\pm 1$. Since the set of points $xB\in \Xw$ with $p_{\j}(x)=0$ is of codimension one in $\Xw$ and $|\Jw|$ is finite, there exists a point $xB\in \Xw$ such that $x_{\j}\not=0$ for any $\j\in \Jw$. 
We will see in Proposition~\ref{prop:genric_point} that such a point $xB$ is generic in $\Xw$.

\begin{example}\label{example_generic_pt_312}
	Let $w=312 \in \mathfrak{S}_3$ in the one-line notation. Then we have 
	$w^{(1)} = (3),~w^{(2)} = (1,3),~w^{(3)} = (1,2,3)$.
	Hence $p_{\i}=0$ on $\Xw$ if $\underline{i} = (2,3)$, and the following points 
	\[
	x =\begin{pmatrix}
	\al & 1 & 0 \\ \be & 0 & 1 \\ 1 & 0 & 0 \end{pmatrix} \qquad (\al, \be\in \C\backslash\{0\})
	\]
	satisfy that $x_{\j} \not=0$ for any $\j\in \Jw=\{(1), (2), (3), (1,2), (1,3)\}$. Hence such a point $xB$ is generic in $X_{312}$.
\end{example}

\begin{remark}
	In the paper~\cite{GGMS87} by I. M. Gelfand, M. Goresky, R. D. MacPherson, and V. V. Serganova, they study 
	the polytopes arising from the closure of torus orbits in the Grasmannian $\textup{Gr}_{k,n}$, including generic points of Schubert varieties in Grasmannian $\textup{Gr}_{k,n}$. 
	These polytopes are matroid polytopes for representable matroids (for more details, see~\cite{BGW03_Coxeter_matroids} and~\cite{White86_Theory_of_matroids}).
\end{remark}

\begin{remark}\label{rmk_moment_map}
The pull-back $\psi^{\ast}(\omega)$ of the Fubini--Study form $\omega$ in the product $\prod_{d=1}^{n-1} \C P^{{n \choose {d} } -1}$ of projective spaces is a symplectic form on $G/B$. With respect to this symplectic form, there is a map called \defi{moment map}
\[
\mu\colon G/B\to \V
\]
such that 
$\mu(uB)=(u^{-1}(1),\dots,u^{-1}(n))$ (see~\cite[\S5.1]{GS87_Combinatorial} and~\cite[\S3]{LMP_BIP}).
Then $\mu(G/B)$ is the convex hull of $\mu(uB)$'s for all $u\in \mathfrak{S}_n$, that is the permutohedron of dimension $n-1$. The image $\mu(\Yw)$ is the convex hull of $\mu(uB)$'s for all $u\le w$ by \cite[Theorem 2]{Atiyah82}. Indeed, we have that
\[
\mu(Y_w) = \text{Conv}\{(u^{-1}(1),\dots,u^{-1}(n)) \mid u \leq w\}
\]
which agrees with \defi{the Bruhat interval polytope $Q_{id, w^{-1}}$} in \cite{ts-wi15}. Hence the fan of $Y_w$ is the normal fan of the polytope $Q_{id,w^{-1}}$. In general, for $v \leq w$, the Bruhat interval polytope $Q_{v,w}$ is defined by
\begin{equation}\label{eq_def_of_Qvw}
Q_{v,w} = \text{Conv}\{(u(1),\dots,u(n)) \mid v \leq u \leq w\}.
\end{equation}
\end{remark}

\section{Fan of the generic torus orbit closure $\Yw$} \label{sect:3}

The set $\Hom(\C^*,T)$ of algebraic homomorphisms from $\C^*=\C\backslash\{0\}$ to $T$ forms an abelian group under the multiplication of $T$ and $\Hom(\C^*,T)$ is isomorphic to $\Z^n$ through the correspondence 
\begin{equation} \label{eq:HomCT}
(a_1,\dots,a_n) \in \Z^n \to \big( t\mapsto (t^{a_1},\dots,t^{a_n})\big)\in \Hom(\C^*,T).
\end{equation}
We identify $\Hom(\C^*,T)$ with $\Z^n$ through \eqref{eq:HomCT} and $\Hom(\C^*,T)\otimes\R$ with $\R^n$ but we often denote the element of $\Hom(\C^*,T)$  corresponding to $\a\in \Z^n$ by $\lambda^\a$.  

\begin{remark} \label{rema:ambient_space}
	Since the action of $T$ on $\Yw$ is not effective, the ambient space of the fan of $\Yw$ is the quotient of $\Hom(\C^*,T)\otimes\R$ by the subspace $\Hom(\C^*,T_w)\otimes\R$, where $T_w$ is the toral subgroup of $T$ which fixes $\Yw$ pointwise ($F_w$ in Theorem~\ref{theo:main1} is $\Hom(\C^*,T_w)\otimes \R$). However, for simplicity, we will think of $\R^n$ as the ambient space of the fan of $\Yw$ throughout the paper.
\end{remark}

We will find the maximal cones in the fan of $\Yw$ using the Orbit-Cone correspondence (see~\cite[Proposition~3.2.2.]{CLS11Toric}). Namely, we observe the limit point $\lim_{t\to 0}\lambda^\a(t)\cdot x$ for a generic point $x\in \Xw$ and $\a\in \UZ$. It is known and not difficult to see that $\Ywo$, where $w_0$ is the longest element of $\mathfrak{S}_n$, is the permutohedral variety of complex dimension $n-1$ and the cones  
\begin{equation}
\Cv:=\{ (a_1,\dots,a_n)\in \U \mid a_{v(1)}\le a_{v(2)}\le \cdots\le a_{v(n)}\}
\end{equation}
for $v\in\mathfrak{S}_n$ are the maximal cones in the fan of $\Ywo$. Unless $w=w_0$, we will see that a maximal cone in the fan of $\Yw$ is the union of some of $\Cv$'s (Corollary~\ref{coro:maximal_cone}).  Here is an example which shows how to find the maximal cones in the fan of $\Yw$.

\begin{example}\label{exam:w312}
	Suppose that $G = \GL_3(\mathbb{C})$ and $w = 312 \in \mathfrak{S}_3$.
	Take the (generic) point $xB \in X_w$ as in Example~\ref{example_generic_pt_312}. 
	Then Example~\ref{example_Plucker_GL3} shows that for $\a= [a_1,a_2,a_3] \in \Z^3$, the corresponding curve in the product $\C P^2 \times \C P^2$ of projective spaces is given by
	\begin{align*}
	\lambda^\a(t) \cdot \psi(xB) &= 
	(t^{a_1}, t^{a_2}, t^{a_3}) \cdot ([\al,\be,1], [-\be, -1, 0])\\
	&= ([t^{a_1}\al, t^{a_2}\be, t^{a_3}],
	[-t^{a_1+a_2}\be, -t^{a_1+a_3}, 0]).
	\end{align*}
	Take $v=321$. Then for any $\a = (a_1,a_2,a_3) \in \text{Int}(C(321)) \cap \Z^3$ (so that $a_3<a_2<a_1$), we have that
	\begin{align*}
	\lim_{t \to 0} \lambda^\a(t) \cdot \psi(xB)
	&= \lim_{t \to 0} ([t^{a_1}\al, t^{a_2}\be, t^{a_3}],
	[-t^{a_1+a_2}\be, -t^{a_1+a_3}, 0]) \\
	&= \lim_{t \to 0}
	([t^{a_1-a_3}\al, t^{a_2-a_3}\be, 1],
	[t^{a_2-a_3}\be, 1, 0]) \\
	&= ([0,0,1],[0,1,0]) \\
	&= \psi(312{B}).
	\end{align*}
	A similar argument shows that the limit point corresponding to each $\Cv$ is as follows:
	\begin{align*}
	C({123}) \colon ([1,0,0],[1,0,0]), \quad
	C({132}) \colon ([1,0,0], [0,1,0]), \quad
	C({213}) \colon ([0,1,0],[1,0,0]),\\
	C({231}) \colon ([0,1,0],[1,0,0]), \quad
	C({312}) \colon ([0,0,1],[0,1,0]), \quad
	C({321})\colon ([0,0,1],[0,1,0]).
	\end{align*}
	Hence there are four limit points and accordingly the fan of $\Yw$ consists of four maximal cones:
	\[
	C({123}),\ C({132}),\ C({213}) \cup C({231}),\ C({312}) \cup C({321}).
	\]
	See Figure~\ref{fig_fan_312}. 
 Here, the ambient space of the fan of $Y_w$ is the quotient space $\R^3/\langle (1,1,1) \rangle$. The identification $\R^3/\langle(1,1,1) \rangle \stackrel{\cong}{\longrightarrow} \R^2$ given by $[a_1,a_2,a_3] \mapsto (a_1-a_2, a_2-a_3)$ is well-defined. Using this identification, we draw Figures~\ref{fig_fan_Fl3} and~\ref{fig_fan_312}. For instance, the cone $C(213)$ consists of points $[a_1,a_2,a_3]$ satisfying $a_2 \leq a_1 \leq a_3$. Hence it corresponds to the set of points $\{(b_1,b_2)\in\R^2 \mid b_1 \geq 0, b_1+b_2 \leq 0\}$ under the identification.
	We note that $\Yw$ is a Hirzebruch surface $\C P^2 \# \overline{\C P^2}$. In this case, the Schubert variety $X_{312}$ itself is a toric variety, i.e., $X_{312} = Y_{312}$. We will discuss when $X_w$ is a toric variety in general in Section~\ref{sect:6}.   
\end{example}
\begin{figure}
	\begin{minipage}{0.55\textwidth}
		\centering
		\begin{tikzpicture}
		\draw(-2,0)--(2,0);
		\draw(0,-2)--(0,2);
		\draw(-2,2)--(2,-2);
		
		\node at (-1,-1) {\small $C(123)$};
		\node at (0.6,-1.3) {\small $C(213)$};
		\node at (1.3,-0.5) {\small $C(231)$};
		\node at (1,1) {\small $C(321)$};
		\node at (-0.6,1.3) {\small $C(312)$};
		\node at (-1.3,0.5) {\small $C(132)$};

		\end{tikzpicture}
		\caption{Cones $\Cv$ for $v \in \mathfrak{S}_3$.} 
		\label{fig_fan_Fl3}
	\end{minipage}~
	\begin{minipage}{0.45\textwidth}
		\centering
		\begin{tikzpicture}

		\filldraw[fill=yellow!40!white, draw opacity = 0] (-2,2)--(2,2)--(2,0)--(0,0)--cycle;
		\filldraw[fill=green!10!white, draw opacity = 0] (-2,2)--(0,0)--(-2,0)--cycle;
		\filldraw[fill=blue!10!white, draw opacity = 0] (-2,0)--(0,0)--(0,-2)--(-2,-2)--cycle;
		\filldraw[fill=purple!10!white, draw opacity = 0] (0,0)--(2,0)--(2,-2)--(0,-2)--cycle;

 		\draw (-2,0)--(2,0);
\draw (0,0)--(0,-2);
\draw (-2,2)--(0,0);
\draw[dotted] (0,2)--(0,0);
\draw[dotted] (0,0)--(2,-2);		
		
		\node at (-1,-1) {\small $C(123)$};
		\node at (1.3,-1) {\small $C(213) \cup C(231)$};
		\node at (0.5,1) {\small $C(312) \cup C(321)$};
		\node at (-1.3,0.5) {\small $C(132)$};

		\end{tikzpicture}
		\caption{The fan of $Y_{312}$.}
		\label{fig_fan_312}
	\end{minipage}
\end{figure}

In the previous example, we saw that any integral point $\mathbf{a}$ in the relative interior of $C(v)$ has the same limit point $\lim_{t \to 0}\lambda^{\mathbf a}(t) \cdot \psi(xB)$. Moreover, each limit point is the moment map image of $uB$ for some $u \in \mathfrak{S}_n$. 
Motivated by this observation, 
we introduce an operation $\Rw \colon \mathfrak{S}_n \to \mathfrak{S}_n$ with respect to $w$.
\begin{definition} 
	Fix $w\in \frak{S}_n$.
	The \defi{retraction} $\Rw \colon \mathfrak{S}_n \to \mathfrak{S}_n$ is defined as follows: for $v\in \mathfrak{S}_n$, we inductively choose 
	\begin{align*}
	i_1 &:= \min\{i \in [n] \mid v(i) \leq w(1) = w^{(1)}\},\\
	i_2 &:= \min\{i \in [n]\setminus \{i_1\} \mid \{ v(i_1), v(i)\}\!\uparrow \leq w^{(2)}\},\\
	\vdots\\
	i_n &:= \min\{i \in [n]\setminus \{i_1,\dots,i_{n-1}\} \mid \{v(i_1),\dots,v(i_{n-1}),v(i)\}\!\uparrow \leq w^{(n)} \},
	\end{align*}
	and define  
	$$\Rw(v) :=v(i_1)v(i_2)\cdots v(i_n)\in \mathfrak{S}_n.$$
	Here, $\{a_1,\dots,a_d\}\!\uparrow$ denotes the ordered $d$-tuple obtained from $\{a_1,\dots,a_d\}$ by arranging its elements in ascending order and $w^{(d)}=\{w(1),\dots,w(d)\}\!\uparrow$.
\end{definition}

In the following, we fix $w \in \mathfrak{S}_n$ and denote
\[
v' := \Rw(v) 
\]
		to simplify the notation.
The retraction $\Rw$  will play an important role in our argument.  As is well-known, the Bruhat order $v \leq w$ on the symmetric group $\mathfrak{S}_n$ satisfies that
\begin{equation}\label{equation_Bruhat_order_equivalet}
v \leq w \iff v^{(d)} \leq w^{(d)}~~ \text{ for all } 1 \leq d \leq n-1
\end{equation}
(see \cite[(3.2.5)]{BL20Singular}), so the following lemma immediately follows from the definition of the operation.

\begin{example}
	Let $w = 3412$ and $v = 4123$. Then $i_1,\dots,i_4$ are given as follows:
	\begin{align*}
	i_1 &= \min\{i \in [4] \mid v(i) \leq w(1)=3\} = 2, \\
	i_2 &= \min\{i \in [4] \setminus \{2\} \mid \{1, v(i)\}\uparrow \leq (3,4) \} = 1, \\
	i_3 &= \min\{i \in [4] \setminus \{1,2\}\mid
	\{1,4,v(i)\} \uparrow \leq (1,3,4) \} = 3, \\
	i_4 &= 4.
	\end{align*}
	Hence $v' = \Rw(v) = v(2)v(1)v(3)v(4) = 1423$. 
\end{example}
\begin{lemma} \label{lemm:property} 
	The retraction $v' = \Rw(v)$ for $v\in \mathfrak{S}_n$ has the following properties:
	\begin{enumerate}
		\item $v'\le w$ for any $v\in \mathfrak{S}_n$,
		\item $v'=v$ if $v\le w$,
		\item $w_0'=w$ where $w_0$ is the longest element of $\mathfrak{S}_n$ as before. 
	\end{enumerate}
\end{lemma}

By Lemma~\ref{lemm:property}, the image of the retraction $\Rw$ is the Bruhat interval $[id,w] := \{u \in \mathfrak{S}_n \mid u \leq w\}$. 
Hence, the set $\mathfrak{S}_n$ retracts to its subset $[id,w]$ via $R_w$ indeed.
The following proposition is a key observation.

\begin{proposition}\label{prop_limit_point_of_Cv}
	Let $x$ be a point in $\Xw$ with $x_{\j} \neq 0$ for any $\j \in \Jw$ {\rm (}see \eqref{eq:\Jw} for $\Jw${\rm )}. Then for any $v\in \mathfrak{S}_n$ and any $\a\in \Int(\Cv)\cap \UZ$, we have 
	\[
	\lim_{t \to 0} \lambda^\a(t) \cdot x = v'B.
	\]
\end{proposition}

\begin{proof}
	Suppose that $v\leq w$. Then, since $v^{(d)}\leq w^{(d)}$ for any $1\leq d\leq n-1$ by~\eqref{equation_Bruhat_order_equivalet}, we have $v^{(d)}\in \Jw$ and hence 
	\[
	x_{v^{(d)}} \neq 0 ~~\text{ for all } 1 \leq d \leq n-1.
	\]
	On the other hand, since $\a=[a_1,\dots,a_n]\in \Int(\Cv)\cap \UZ$, we have $a_{v(1)}<a_{v(2)}<\cdots<a_{v(n)}$. Therefore the sum $\sum_{k=1}^d a_{v(k)}$ is smallest among the sum of arbitrary $d$ elements in $a_1,\dots,a_n$. Then, the same argument as in Example~\ref{exam:w312} shows that 
	$$\lim_{t\to 0}\lambda^{\a}(t)\cdot x=vB.$$ 
	Here $v'=v$ by Lemma~\ref{lemm:property} because $v\leq w$. This proves the proposition when $v\leq w$. 
	
	Suppose that $v\nleq w$. Since $v'\le w$ by Lemma~\ref{lemm:property}, ${(v')}^{(d)}\in \Jw$ and hence 
	\[
	x_{(v')^{(d)}} \neq 0 ~~\text{ for all } 1 \leq d \leq n-1.
	\]
	By \eqref{eq:\Iw} we have $x_{\i}=0$ for $\i\in I_{d,n}$ with $\i\nleq w^{(d)}$, so the construction of $v'$ shows that the sum $\sum_{k=1}^d a_{v'(k)}$ is smallest among the sum $\sum_{k\in \i}a_k$ for $\i\in I_{d,n}$ with $x_{\i}\not=0$. 
	This implies the desired identity in the proposition as before.
\end{proof}

The Orbit-Cone correspondence and Proposition~\ref{prop_limit_point_of_Cv} imply the following.

\begin{corollary} \label{coro:maximal_cone}
	The maximal cone in the fan of $\Yw$ corresponding to a fixed point $uB$ $(u\le w)$ is of the form
	\begin{equation*}
	\Cbwu:=\bigcup_{v \in \mathfrak{S}_n \text{ s.t. } v' = u}\Cv.
	\end{equation*}
\end{corollary}

Using Proposition~\ref{prop_limit_point_of_Cv},  we can characterize generic points in $\Xw$ as follows.
\begin{proposition} \label{prop:genric_point}
	A point $x$ in $\Xw$ is generic in $\Xw$ if and only if $x_{\j} \neq 0$ for any $\j \in \Jw$.
	{\rm (}Therefore, there exists a generic point in $X_w$ as explained after~\eqref{eq:\Jw}.{\rm )}  
\end{proposition}
\begin{proof}
	Suppose that $x_{\j}\not=0$ for any $\j\in \Jw$.  Then 
	Proposition~\ref{prop_limit_point_of_Cv} shows that the orbit closure $\overline{T\cdot x}$ contains the fixed point $v'B$ for any $v$ but $v'=v$ if $v\le w$ by Lemma~\ref{lemm:property} (2). Hence $\overline{T\cdot x}$ contains all the $T$-fixed points in $\Xw$, which means that $x$ is generic in $\Xw$, proving the \lq\lq if\rq\rq part in the proposition.
	
	Suppose that $x$ is a generic point in $\Xw$ but $x_{\j}=0$ for some $\j\in \Jw$.  Then the $\j$-th Pl\"{u}cker coordinate of the limit point $\lim_{t\to 0}\lambda^{\a}(t)\cdot x$ vanishes for any $\a\in \UZ$, which means that the $\j$-th coordinate of any $T$-fixed point in $\overline{T\cdot x}$ vanishes.  On the other hand, there is some $v\le w$ with $v^{(d)}=\j$ by definition of $\Jw$ where $d$ is the length of $\j$.  Since $v^{(k)}$-th Pl\"{u}cker coordinate of the $T$-fixed point $vB$ does not vanish for any $1\le k\le n$, this shows that $vB\notin \overline{T\cdot x}$ although $v\le w$.  This contradicts $x$ being generic in $\Xw$.  Therefore $x_{\j}\not=0$ for any $\j\in \Jw$, proving the \lq\lq only if \rq\rq part in the proposition. 
\end{proof}

\section{Right weak order}
\label{sec_right_weak_order}

In this section, for a fixed $w \in \mathfrak{S}_n$, we describe the set $\{v \in \mathfrak{S}_n \mid v' = u\}$ in Corollary~\ref{coro:maximal_cone} using the right weak (Bruhat) order on {$\mathfrak{S}_n$} (see \cite[\S 3.1]{BB05Combinatorics}). For $u_1, u_2 \in {\mathfrak{S}_n}$, $u_1 \leq_R u_2$ means that $u_2 = u_1 s_{i_1} \cdots s_{i_k}$ for some simple reflections $s_{i_1},\dots, s_{i_k}$ such that $\ell(u_1 s_{i_1} \dots s_{i_j}) = \ell(u_1) + j$ for $0 \leq j \leq k$, 
where $\ell(v)$ denotes the length of $v\in \mathfrak{S}_n$.
This order is called the \defi{right weak order}. 
The right weak order interval $[u_1,u_2]_R$ is defined to be the set $[u_1,u_2]_R := \{x \in {\mathfrak{S}_n} \mid u_1 \leq_R x \leq_R u_2\}$. 
\begin{lemma} \label{lemm:vprime}
	$v\ge_R v'$ for any $v\in\frak{S}_n$.
\end{lemma}

\begin{proof}
	There is $0\le d\le n$ such that $v(i)= v'(i)$ for $1\le i\le d$. 
	We shall prove the lemma by downward induction on $d$ starting from $n$. If $d=n$, then $v= v'$ and the lemma trivially holds in this case. 
	We assume that the lemma holds for $v$ with $v(i)=v'(i)$ for $1\le i\le d$ and prove the lemma for $v$ with $v(i)=v'(i)$ for $1\le i\le d-1$. We may assume $v(d)\not= v'(d)$. Define $j\in [n]$ by $v(j)=  v'(d)$. Since $v(i)=v'(i)$ for $1\le i\le d-1$, we have $d<j$ and it follows from the construction of $v'$ that 
	\begin{equation} \label{eq:4-1}
	\begin{split}
	&\{ v(1),\dots,v(d-1),v(i)\}\uparrow \nleq w^{(d)}\qquad \text{for $d\le i<j$},\\
	&\{ v(1),\dots,v(d-1),v(j)\}\uparrow \leq w^{(d)}.
	\end{split}
	\end{equation}
	In particular
	\begin{equation} \label{eq:4-2}
	\text{$v(i)>v(j)$ for $d\le i<j$.}
	\end{equation} 
	We consider $u\in \frak{S}_n$ defined by 
	\[
	u=v(1)\cdots v(d-1)v(j)v(d)\cdots v(j-1)v(j+1)\cdots v(n).
	\]
	Since $v(i)=v'(i)$ for $1\le i\le d-1$, it follows from \eqref{eq:4-1} and \eqref{eq:4-2} that we have 
	\[
	\begin{split}
	& u(i)=v'(i)\quad \text{for $1\le i\le d$},\\
	&u'=v',\\
	&v=us_ds_{d+1}\cdots s_{j-1}\quad\text{with $\ell(v)=\ell(u)+(j-d)$, so $v\ge_R u$}.
	\end{split}
	\] 
	Because of the first and second identities above, the induction assumption can be applied to $u$, so that $u\ge_R u'$. This together with the second and last identities above shows $v\ge_R v'$. 
\end{proof}

\begin{lemma}\label{lemm:retraction_interval}
	If $v' \le_R z\le_R v$, then $z'=v'$.  
\end{lemma}
\begin{proof}
	Since $z\le_R v$, we have $v=zs_{i_1}s_{i_2}\cdots s_{i_p}$ for $p=\ell(v)-\ell(z)$.  Therefore, if we set $z_q=zs_{i_1}s_{i_2}\cdots s_{i_q}$ for $1\le q\le p$, then $z_{q}=z_{q-1}s_{i_q}$ (where $z_0=z$) and 
	\[
	z\le_R z_1\le_R z_2\le_R\cdots \le_R z_{p-1}\le_R z_p=v.
	\]
	Therefore, it suffices to prove the claim when $v=zs_k$ for some $k$.  Then we have 
	\begin{equation} \label{eq:1}
\arraycolsep=1.4pt
	\begin{array}{rcccccccc}
	v&=v(1)&\cdots &v(k-1)&v(k)&v(k+1)&v(k+2)&\cdots& v(n),\\
	z&=v(1)&\cdots &v(k-1)&v(k+1)&v(k)&v(k+2)&\cdots& v(n).
	\end{array}
	\end{equation}
	Since $vs_k=z\le_R v$, we have $v(k)>v(k+1)$; so $z(k)=v(k+1)<v(k)=z(k+1)$, that is, $(z(k),z(k+1))$ is not an inversion in $z$.  On the other hand, since $v'\le_R z$, an inversion in $v'$ is also an inversion in $z$. 
	This means that
	\[
	\text{$(v(k+1)=z(k),v(k)=z(k+1))$ is not an inversion in $v'$,}
	\] 
	i.e., $v(k+1)$ appears ahead of $v(k)$ in the one-line notation of $v'$. This together with \eqref{eq:1} implies $v'=z'$. 
\end{proof}

\begin{proposition}\label{prop_set_of_limit_points}
	Let $u\le w$. Then there exists a {\rm (}unique{\rm )} $\uw\in\frak{S}_n$ such that 
	\[
	\{v\in \frak{S}_n\mid v'=u\}=[u,\uw]_R.
	\]
\end{proposition}

\begin{proof}
	We denote the left hand side in the proposition by $S(u)$. 
By Lemma~\ref{lemm:retraction_interval}, we know that $S(u)$ is the union of certain right weak order intervals. To prove that $S(u)$ is a right weak order interval, it suffices to prove that if 
	\begin{enumerate}
		\item \text{$\ell(vs_p)=\ell(vs_q)=\ell(v)+1$ for some $ p<q$, and}, 
		\item \text{$v'=(vs_p)'=(vs_q)'=u$,}
	\end{enumerate}
	then there exists $\tilde{v}\in \frak{S}_n$ such that {$\{ vs_p,vs_q\}\subset [v,\tilde{v}]_R\subset S(u)$}. 
	We note that 
	\begin{equation} \label{eq:5-1}
	v(p)<v(p+1),\quad v(q)<v(q+1)
	\end{equation}
	by (1) above. We define $i,j\in [n]$ by 
	\begin{equation*} 
	v(p)=u(i),\quad v(q)=u(j).
	\end{equation*}
	Then it follows from (2) above that 
	\begin{equation} \label{eq:5-2}
	\begin{split}
	&\{u(1),\dots,u(i-1), v(p)\}\uparrow \leq w^{(i)},\quad \{u(1),\dots,u(i-1), v(p+1)\}\uparrow \nleq w^{(i)},\\
	&\{u(1),\dots,u(j-1), v(q)\}\uparrow \leq w^{(j)},\quad \{u(1),\dots,u(j-1), v(q+1)\}\uparrow \nleq w^{(j)}.
	\end{split}
	\end{equation}
	We consider two cases. 
	
	Case 1. The case where $q-p\ge 2$. In this case, $s_p$ and $s_q$ commute and we take $\tilde{v}=vs_ps_q$. Then $[v,\tilde{v}]_R=\{v, vs_p, vs_q, \tilde{v}=vs_ps_q\}$, and (2) above implies $\tilde{v}'=u$. Therefore $\{ vs_p,vs_q\}\subset [v,\tilde{v}]_R\subset S(u)$. 
	
	Case 2. The case where $q-p=1$, i.e., $q=p+1$. In this case, it follows from~\eqref{eq:5-1} that 
	\begin{equation} \label{eq:5-3}
	v(p)<v(p+1)<v(p+2).
	\end{equation}
	Note that $i<j$ since $v'=u$ and $u(i)=v(p)<v(p+1)=u(j)$. We take $\tilde{v}=vs_ps_{p+1}s_p$, i.e., 
	\begin{equation} \label{eq:5-4}
	\tilde{v}=v(1)\cdots v(p-1)v(p+2)v(p+1)v(p)v(p+3)\cdots v(n).
	\end{equation}
	{Then 
		\[ [v,\tilde{v}]_R=\{v,\ vs_p,\ vs_{p+1},\ vs_ps_{p+1},\ vs_{p+1}s_p,\ \tilde{v}=vs_ps_{p+1}s_p\}. \]}
	Since $s_ps_{p+1}s_p=s_{p+1}s_ps_{p+1}$ and $q=p+1$, we have $\tilde{v}\ge_R vs_p$ and $\tilde{v}\ge_R vs_q$ by~\eqref{eq:5-3}. By~\eqref{eq:5-2} and \eqref{eq:5-3} we have 
	\[
	\{u(1),\dots,u(i-1), v(p+1)\}\uparrow \nleq w^{(i)},\quad \{u(1),\dots,u(i-1), v(p+2)\}\uparrow \nleq w^{(i)}. 
	\]
	This together with the first inequality in~\eqref{eq:5-2}, \eqref{eq:5-4} and the assumption $v'=u$ show that $\tilde{v}'(i)=v(p)=u(i)$. Then, the conclusion $\tilde{v}'=u$ follows from the second line in~\eqref{eq:5-2} and the assumption $v'=u$. Since 
	\[
	\arraycolsep=1.4pt
	\begin{array}{rccccccccc}
	vs_ps_{p+1}&=v(1)&\cdots &v(p-1)&v(p+1)&v(p+2)&v(p)&v(p+3)&\cdots &v(n),\\
	vs_{p+1}s_p&=v(1)&\cdots &v(p-1)&v(p+2)&v(p)&v(p+1)&v(p+3)&\cdots &v(n),
	\end{array}
	\]
	the same observation for $\tilde{v}$ shows that $(vs_ps_{p+1})'=u=(vs_{p+1}s_p)'$. Therefore $\{ vs_p,vs_q\}\subset [v,\tilde{v}]_R\subset S(u)$ in this case, too. 
\end{proof}

\begin{example}\label{example_1342}
	Take $w=1342$. Then there are four elements $u\in \frak{S}_4$ such that $u\le w$ (see Figure~\ref{fig_Bruhat_order}) and one can check the following (see Figure~\ref{fig_weak_Bruhat_order}):

	\begin{table}[htb]
		\centering
		\begin{tabular}{|c||c|c|c|c|c|c|c|c|c|c|} \hline
			$u$ & 1342 & 1243 &1324 & 1234 \cr
			\hline $u_w$ & 4321& 4231 &3241 & 2341 \cr
			\hline
		\end{tabular} \vspace{1.25ex}
		\caption{$u_w$ when $w=1342$.}
		\label{table1}
	\end{table}
	\noindent
	\[
	\begin{split}
	[1342,4321]_R&=\{ 1342, 3142, 1432, 4132, 3412, 4312, 3421, 4321\},\\
	[1243,4231]_R&=\{ 1243, 2143, 1423, 4123, 2413, 4213, 2431, 4231\},\\
	[1324,3241]_R&=\{ 1324, 3124, 3214, 3241\},\\
	[1234,2341]_R&=\{ 1234, 2134, 2314, 2341\}.
	\end{split}
	\]
	\begin{figure}[H]
	\begin{subfigure}[t]{0.49\textwidth}
		\centering
		\begin{tikzpicture}
		\tikzstyle{every node}=[font=\footnotesize]
		\matrix [matrix of math nodes,column sep={0.48cm,between origins},
		row sep={1cm,between origins},
		nodes={circle, draw, inner sep = 0pt , minimum size=1.2mm}]
		{
			& & & & & \node[label = {above:{4321}}] (4321) {} ; & & & & & \\
			& & & 
			\node[label = {above left:4312}] (4312) {} ; & & 
			\node[label = {above left:4231}] (4231) {} ; & & 
			\node[label = {above right:3421}] (3421) {} ; & & & \\
			& \node[label = {above left:4132}] (4132) {} ; & & 
			\node[label = {left:4213}] (4213) {} ; & & 
			\node[label = {above:3412}] (3412) {} ; & & 
			\node[label = {[label distance = 0.1cm]0:2431}] (2431) {} ; & & 
			\node[label = {above right:3241}] (3241) {} ; & \\
			\node[label = {left:1432}] (1432) {} ; & & 
			\node[label = {left:4123}] (4123) {} ; & & 
			\node[label = {[label distance = 0.01cm]180:2413}] (2413) {} ; & & 
			\node[label = {[label distance = 0.01cm]0:3142}] (3142) {} ; & & 
			\node[label = {right:2341}] (2341) {} ; & & 
			\node[label = {right:3214}] (3214) {} ; \\
			& \node[label = {below left:1423}] (1423) {} ; & & 
			\node[label = {[label distance = 0.1cm]182:1342}, fill=black] (1342) {} ; & & 
			\node[label = {below:2143}] (2143) {} ; & & 
			\node[label = {right:3124}] (3124) {} ; & & 
			\node[label = {below right:2314}] (2314) {} ; & \\
			& & & \node[label = {below left:1243}, fill=black] (1243) {} ; & & 
			\node[label = {[label distance = 0.01cm]190:1324}, fill=black] (1324) {} ; & & 
			\node[label = {below right:2134}] (2134) {} ; & & & \\
			& & & & & \node[label = {below:1234}, fill=black] (1234) {} ; & & & & & \\
		};
		
		\draw (4321)--(4312)--(4132)--(1432)--(1423)--(1243)--(1234)--(2134)--(2314)--(2341)--(3241)--(3421)--(4321);
		\draw (4321)--(4231)--(4132);
		\draw (4231)--(3241);
		\draw (4231)--(2431);
		\draw (4231)--(4213);
		\draw (4312)--(4213)--(2413)--(2143)--(3142)--(3241);
		\draw (4312)--(3412)--(2413)--(1423)--(1324)--(1234);
		\draw (3421)--(3412)--(3214)--(3124)--(1324);
		\draw (3421)--(2431)--(2341)--(2143)--(2134);
		\draw (4132)--(4123)--(1423);
		\draw (4132)--(3142)--(3124)--(2134);
		\draw (4213)--(4123)--(2143)--(1243);
		\draw (4213)--(3214);
		\draw (3412)--(1432)--(1342)--(1243);
		\draw (2431)--(1432);
		\draw (2431)--(2413)--(2314);
		\draw (3142)--(1342)--(1324);
		\draw (4123)--(3124);
		\draw (2341)--(1342);
		\draw (2314)--(1324);
		\draw (3412)--(3142);
		\draw (3241)--(3214)--(2314);
		\end{tikzpicture}
		\caption{Bruhat order of $\mathfrak{S}_4$ and elements smaller than or equal to $1342$ are marked by $\bullet$.}
		\label{fig_Bruhat_order}
	\end{subfigure}~\hspace{1em}
	\begin{subfigure}[t]{0.49\textwidth}
		\centering
		\begin{tikzpicture}
		\tikzstyle{every node}=[font=\footnotesize]
		\matrix [matrix of math nodes,column sep={0.48cm,between origins},
		row sep={1cm,between origins},
		nodes={circle, draw, inner sep = 0pt , minimum size=1.2mm}]
		{
			& & & & & \node[label = {above:4321}, fill=red!40!white] (4321) {} ; & & & & & \\
			& & & \node[label = {above left:4312}, fill=red!40!white] (4312) {} ; & & 
			\node[label = {above left:4231}, fill=blue!40!white] (4231) {} ; & & 
			\node[label = {above right:3421}, fill=red!40!white] (3421) {} ; & & & \\
			& \node[label = {above left:4132}, fill=red!40!white] (4132) {} ; & & 
			\node[label = {left:4213}, fill=blue!40!white] (4213) {} ; & & 
			\node[label = {above:3412}, fill=red!40!white] (3412) {} ; & & 
			\node[label = {[label distance = 0.1cm]0:2431}, fill=blue!40!white] (2431) {} ; & & 
			\node[label = {above right:3241}, fill=green!50!white] (3241) {} ; & \\
			\node[label = {left:1432}, fill=red!40!white] (1432) {} ; & & 
			\node[label = {left:4123}, fill=blue!40!white] (4123) {} ; & & 
			\node[label = {[label distance = 0.01cm]180:2413}, fill=blue!40!white] (2413) {} ; & & 
			\node[label = {[label distance = 0.01cm]0:3142}, fill=red!40!white] (3142) {} ; & & 
			\node[label = {right:2341}, fill=purple!40!white] (2341) {} ; & & 
			\node[label = {right:3214}, fill=green!50!white] (3214) {} ; \\
			& \node[label = {below left:1423}, fill=blue!40!white] (1423) {} ; & & 
			\node[label = {[label distance = 0.1cm]182:1342}, fill=red!40!white] (1342) {} ; & & 
			\node[label = {below:2143}, fill=blue!40!white] (2143) {} ; & & 
			\node[label = {right:3124}, fill=green!50!white] (3124) {} ; & & 
			\node[label = {below right:2314}, fill=purple!40!white] (2314) {} ; & \\
			& & & \node[label = {below left:1243}, fill=blue!40!white] (1243) {} ; & & 
			\node[label = {[label distance = 0.01cm]190:1324}, fill=green!50!white] (1324) {} ; & & 
			\node[label = {below right:2134}, fill=purple!40!white] (2134) {} ; & & & \\
			& & & & & \node[label = {below:1234}, fill=purple!40!white] (1234) {} ; & & & & & \\
		};
		
		\draw[line width=1ex, red,nearly transparent] (4321)--(4312)--(3412)--(3421)--(4321);
		\draw[line width=1ex, red,nearly transparent] (4312)--(4132)--(1432)--(1342)--(3142)--(3412);
		
		\draw[line width=1ex, blue,nearly transparent] (4213)--(4231)--(2431)--(2413)--(4213);
		\draw[line width=1ex, blue,nearly transparent] (4213)--(4123)--(1423)--(1243)--(2143)--(2413);
		
		\draw[line width=1ex, green!50!black,nearly transparent] (3241)--(3214)--(3124)--(1324);
		\draw[line width=1ex, purple,nearly transparent] (2341)--(2314)--(2134)--(1234);
		
		\draw (4321)--(4312)--(4132)--(1432)--(1423)--(1243)--(1234)--(2134)--(2314)--(2341)--(3241)--(3421)--(4321);
		\draw (4321)--(4231)--(4213)--(4123)--(1423);
		\draw (4132)--(4123);
		\draw (4213)--(2413)--(2143)--(1243);
		\draw (4312)--(3412)--(3142)--(1342)--(1432);
		\draw (3421)--(3412);
		\draw (4231)--(2431)--(2413);
		\draw (2431)--(2341);
		\draw (3241)--(3214)--(3124)--(1324)--(1234);
		\draw (3142)--(3124);
		\draw (3214)--(2314);
		\draw (1342)--(1324);
		\draw (2143)--(2134);
		
		\end{tikzpicture}
		\caption{Right weak order of $\mathfrak{S}_4$ and right weak order intervals in Table~\ref{table1}.}
		\label{fig_weak_Bruhat_order}
	\end{subfigure}
	\caption{Bruhat order and right weak order of $\mathfrak{S}_4$.}
\end{figure}
\end{example}

	\begin{remark}
		One can observe from Example~\ref{example_1342} that for an element $v \in \mathfrak{S}_n$, its  retraction $v'$ is the point in $[id,w]$ which is closest to $v$ with respect to the metric $d(x,y) := \ell(x^{-1}y)$ defined on $\mathfrak{S}_n$ (see~\cite{LMP_metric} for further study). 
	\end{remark}
\section{Dual cone of $C_w(w)$} \label{sect:4}
{By Corollary~\ref{coro:maximal_cone} and Proposition~\ref{prop_set_of_limit_points}, the maximal cone in the fan of $\Yw$ corresponding to the fixed point $uB$ $(u\le w)$ is of the form
	\[
	\Cbwu = \bigcup_{v \in [u, \uw]_R} \Cv.
	\]
	As noted in Lemma~\ref{lemm:property} $w_0'=w$, so $u_w=w_0$ when $u=w$. 
	Hence the maximal cone corresponding to $wB$ is of the form} 
\[
\Cbw:= C_w(w) = \bigcup_{v \in [w, w_0]_R} \Cv.
\]
Our purpose of this section is to identify the dual of the maximal cone $\Cbw$ (Proposition~\ref{prop:dual}).

\begin{definition} \label{defi:E(w)}
	For $w\in \mathfrak{S}_n$, we define 
	\[
	\Refw:=\{ (w(i),w(j))\mid 1\le i<j\le n,\ \ell(w)-\ell(t_{w(i),w(j)}w)=1\}
	\]
	where $t_{a,b}$ denotes the transposition of $a$ and $b$ and $\ell(v)$ denotes the length of a permutation $v$ as before. 
	(Note. The condition $\ell(w)-\ell(t_{w(i),w(j)}w)=1$ above is equivalent to $w(i)>w(j)$ and $w(k)\notin [w(j),w(i)]$ for $i<\forall k<j$.)
\end{definition}

\begin{example}\label{example_Rw}
	\begin{enumerate}
		\item If $w=3152674$, then 
		\[
		\Refw=\{(3,1), (3,2), (5,2), (5,4), (6,4), (7,4)\}.
		\]
		\item If $w=3715264$, then 
		\[
		\Refw=\{ (3,1), (3,2), (7,1), (7,5), (7,6), (5,2), (5,4), (6,4)\}.
		\]
	\end{enumerate}
\end{example}

We define
\begin{equation} \label{eq:Dw}
\Dw:= \text{ the cone in $\V$ spanned by } \{e_b-e_a\mid (a,b)\in \Refw\}.
\end{equation}

Since the cone $\Dw$ is generated by vectors of the form $e_b - e_a$ as in~\eqref{eq:Dw}, its dual $\Dw^{\vee}$ consists of points $x = (x_1,\dots,x_n) \in \U$ satisfying that 
\[
\langle e_b - e_a, x\rangle = x_b - x_a \geq 0
\]
for all $(a,b) \in \Refw$. Moreover, the dual cone $\Dw^{\vee}$ can be described as follows:
\begin{proposition} \label{prop:dual}
	$\Dw^\vee=\Cbw (=\bigcup_{v \in [w, w_0]_R} C(v))$.
\end{proposition}

This proposition follows from the following two lemmas. {Remember that $u_w=w_0$ when $u=w$ and hence $v\in [w,w_0]_R$ if and only if $v'=w$ by Proposition~\ref{prop_set_of_limit_points}.} 

\begin{lemma}
	If $v\in [w,w_0]_R$, i.e., $v'=w$, then $\Cv$ is contained in $\Dw^\vee$.
\end{lemma}

\begin{proof}
	Since $v\in [w,w_0]_R$, one can write
	\[
	v=ws_{i_1}\cdots s_{i_k}
	\]
	where $\ell(v)=\ell(w)+k$. We prove the lemma by induction on $k$. 
	If $k=0$, then $v=w$ and the lemma is obvious. Suppose that $k\ge 1$ and the lemma holds for $k-1$. We set 
	\[
	u=ws_{i_1}\cdots s_{i_{k-1}} \qquad\text{and}\qquad i_k=p.
	\]
	Then, since $v=us_p$, we have
	\begin{equation} \label{eq:vu}
	\begin{split}
	\arraycolsep=1.4pt
	&\begin{array}{rcccccccl}
	v&=v(1)&\cdots &v(p-1)&v(p)&v(p+1)&v(p+2)&\cdots& v(n),\\
	u&=v(1)&\cdots &v(p-1)&v(p+1)&v(p)&v(p+2)&\cdots &v(n),\quad\text{and}
	\end{array} \\
	&\ell(v)=\ell(u)+1.
	\end{split}
	\end{equation}
	Take any element $(a,b)$ of $\Refw$. By induction assumption we have 
	\[
	\langle e_b-e_a, x\rangle \ge 0\qquad (\forall x\in \Cu)
	\]
	where $\langle\ ,\ \rangle$ denotes the standard inner product on $\R^n$. The above inequality is equivalent to $a$ being ahead of $b$ in the one-line notation for $u$. What we have to prove is that $a$ is still ahead of $b$ in the one-line notation for $v$. We take two cases. 
	
	Case 1. The case where $\{v(p),v(p+1)\}\not=\{a,b\}$. In this case, it is easy to see from \eqref{eq:vu} that $a$ is still ahead of $b$ in the one-line notation for $v$ since so is for $u$. 
	
	Case 2. The case where $\{v(p),v(p+1)\}=\{a,b\}$. 
	Since $a$ is ahead of $b$ in the one-line notation for $u$, we have $a=v(p+1)$ and $b=v(p)$ by \eqref{eq:vu} in this case. We note that $a>b$ since $(a,b)\in \Refw$. This together with $v=us_p$ shows that $\ell(v)=\ell(u)-1$ but this contradicts the last identity in \eqref{eq:vu}. Therefore Case 2 does not occur. 
\end{proof}

\begin{lemma}
	If $v\notin [w,w_0]_R$, i.e., $v'\not=w$, then $\Cv$ is not contained in $\Dw^\vee$.
\end{lemma}

\begin{proof}
	Since $v'<w$ {by assumption and Lemma~\ref{lemm:property} (1)}, there exists $i\in [n]$ such that 
	\begin{equation} \label{eq:vprime}
	v'(j)=w(j)\quad (1\le j\le i-1),\qquad v'(i)<w(i).
	\end{equation}
	Define $q$ and $m$ by 
	\begin{equation} \label{eq:qm}
	v'(i)=v(q)=w(m).
	\end{equation}
	It follows from \eqref{eq:vprime} that 
	\begin{equation} \label{eq:im}
	i< m
	\end{equation}
	and from the construction of $v'$, \eqref{eq:vprime} and \eqref{eq:qm} that 
	\begin{equation} \label{eq:*}
	\text{any element in }\{v(1),\dots,v(q-1)\}\backslash\{w(1),\dots,w(i-1)\} >w(i).
	\end{equation}
	By~the second inequality in~\eqref{eq:vprime} and \eqref{eq:qm}, we have that $w(i) > w(m)$. 
	Since $i<m$ by \eqref{eq:im}, there exists $\ell$ such that 
	\begin{equation} \label{eq:well}
	i \leq \ell < m, \qquad w(i)\ge w(\ell)>w(m),\qquad \text{and}\qquad (w(\ell),w(m))\in \Refw. 
	\end{equation}
	Now we define $p$ by $v(p)=w(\ell)$. Then, since $i \le \ell$, we have 
	\begin{equation} \label{eq:vp}
	v(p)=w(\ell)\notin\{w(1),\dots,w(i-1)\}.
	\end{equation} 
	It follows from \eqref{eq:qm}, \eqref{eq:*}, \eqref{eq:well} and \eqref{eq:vp} that 
	\begin{equation*} \label{eq:qp}
	q<p \qquad\text{and}\qquad v(q)=w(m)<w(\ell)=v(p).
	\end{equation*}
	Since $q<p$, $e_{v(q)}-e_{v(p)}$ takes a negative value on $\Cv$ (through the inner product) but $(v(p),v(q))=(w(\ell),w(m))\in \Refw$ by \eqref{eq:well}. This shows that $\Cv$ is not contained in $\Dw^\vee$, proving the lemma. 
\end{proof}

The two lemmas above imply Proposition~\ref{prop:dual}.

\section{Simpliciality of $\Dw$ and graph $\Gw$} \label{sect:5}
Let $\Gw$ be the graph associated to $\Refw$, i.e., the vertices of $\Gw$ are the positive integers appearing in $\Refw$ (so the vertices of $\Gw$ are elements in $[n]$) and elements in $\Refw$ are the edges. 
We show that $\Yw$ is smooth at $wB$ if and only if $\Gw$ is a forest (Theorem~\ref{thm_Dw_simplicial_iff_Gw_forest}). 
\begin{example}
	Following Example~\ref{example_Rw}, we have graphs $\G_{3152674}$ and
	$\G_{3715264}$ as in Figures~\ref{graph_3152674} and \ref{graph_3715264}.
\end{example}
\begin{figure}[t]
	\centering 
	\begin{minipage}[b]{0.5\textwidth}
		\centering
		\begin{tikzpicture}[scale = 0.5]
		\draw[->, gray, very thin] (0,0) -- (7.5,0) node[very near end, sloped, below] {position};
		\draw[->, gray, very thin] (0,0) -- (0,7.5) node[very near end, sloped, above] {value};
		\begin{scope}[color=gray!50, thin]
		\foreach \xi in {1,2,3,4,5,6,7} {\draw (\xi,0) -- (\xi, 7);}%
		\foreach \yi in {1,2,3,4,5,6,7} {\draw (0,\yi) -- (7,\yi);}
		\end{scope}
		
		\node (3) [circle,draw, fill=white!20, inner sep = 0.25mm] at (1,3) {3};
		\node (1) [circle,draw, fill=white!20, inner sep = 0.25mm] at (2,1) {1};
		\node (5) [circle,draw, fill=white!20, inner sep = 0.25mm] at (3,5) {5};
		\node (2) [circle,draw, fill=white!20, inner sep = 0.25mm] at (4,2) {2};
		\node (6) [circle,draw, fill=white!20, inner sep = 0.25mm] at (5,6) {6};
		\node (7) [circle,draw, fill=white!20, inner sep = 0.25mm] at (6,7) {7};
		\node (4) [circle,draw, fill=white!20, inner sep = 0.25mm] at (7,4) {4};
		
		\draw[->-=.5] (3)--(1);
		\draw[->-=.5] (3)--(2);
		\draw[->-=.5] (5)--(2);
		\draw[->-=.5] (5)--(4);
		\draw[->-=.5] (6)--(4);
		\draw[->-=.5] (7)--(4);
		\end{tikzpicture}
		\caption{Graph $\Gamma_{3152674}$.}
		\label{graph_3152674}
	\end{minipage}%
	\begin{minipage}[b]{0.5\textwidth}
		\centering
		\begin{tikzpicture}[scale = 0.5]
		\draw[->, gray, very thin] (0,0) -- (7.5,0) node[very near end, sloped, below] {position};
		\draw[->, gray, very thin] (0,0) -- (0,7.5) node[very near end, sloped, above] {value};
		\begin{scope}[color=gray!50, thin]
		\foreach \xi in {1,...,7} {\draw (\xi,0) -- (\xi, 7);}%
		\foreach \yi in {1,...,7} {\draw (0,\yi) -- (7,\yi);}
		\end{scope}
		
		\node (3) [circle,draw, fill=white!20, inner sep = 0.25mm] at (1,3) {3};
		\node (7) [circle,draw, fill=white!20, inner sep = 0.25mm] at (2,7) {7};
		\node (1) [circle,draw, fill=white!20, inner sep = 0.25mm] at (3,1) {1};
		\node (5) [circle,draw, fill=white!20, inner sep = 0.25mm] at (4,5) {5};
		\node (2) [circle,draw, fill=white!20, inner sep = 0.25mm] at (5,2) {2};
		\node (6) [circle,draw, fill=white!20, inner sep = 0.25mm] at (6,6) {6};
		\node (4) [circle,draw, fill=white!20, inner sep = 0.25mm] at (7,4) {4};
		
		\draw[->-=.5] (3)--(1);
		\draw[->-=.5] (3)--(2);
		\draw[->-=.5] (5)--(2);
		\draw[->-=.5] (5)--(4);
		\draw[->-=.5] (6)--(4);
		\draw[->-=.5] (7)--(1);
		\draw[->-=.5] (7)--(5);
		\draw[->-=.5] (7)--(6);
		\end{tikzpicture}
		\caption{Graph $\Gamma_{3715264}$.}
		\label{graph_3715264}
	\end{minipage}
\end{figure}
\begin{lemma} \label{lemm:edge}
	If $(a,b)\in \Refw$, then $e_b-e_a$ is an edge vector of $\Dw$. 
\end{lemma}

\begin{proof}
	Suppose that 
	\begin{equation} \label{eq:ba}
	e_b-e_a=\sum_{i=1}^k c_i(e_{b_i}-e_{a_i}) \qquad \text{with $c_i>0$}
	\end{equation}
	where $(a,b)\not=(a_i,b_i)\in \Refw$ for $i=1,\dots,k$. Suppose $k\ge 2$ and we deduce a contradiction. 
	Since $(a_i,b_i)\in\Refw$, we have $a_i>b_i$. Therefore, it follows from \eqref{eq:ba} that $\displaystyle{a=\max_{1\le i\le k}\{a_i\}}$. We may assume $a_1=a$ if necessary by changing indices. Since $(a_1,b_1)\not=(a,b)$ and $a_1=a$, we have $b_1\not=b$. In order for \eqref{eq:ba} to hold, $e_{b_1}$ must be killed at the right hand side of \eqref{eq:ba}. This means that there exists some $i\not=1$ such that $a_i=b_1$. We may assume $a_2=b_1$ if necessary by changing indices. If $b_2\not=b$, then we repeat the same argument and may assume that $a_3=b_2$. We repeat this argument. Then we reach $m\ge 2$ such that $b_m=b$, i.e., we obtain the following sequence of pairs in $\Refw$: 
	\begin{equation} \label{eq:sequence}
	\begin{split}
	&(a_1,b_1),\ (a_2,b_2),\ \dots,\ (a_m,b_m)\\
	&\text{such that $a_1=a$, $b_m=b$, $b_i=a_{i+1}$ for $1\le i\le m-1$}.
	\end{split}
	\end{equation}
	Remember that since $(a_i,b_i)\in\Refw$, $a_i$ is ahead of $b_i$ in the one-line notation for~$w$. Since $m\ge 2$, this together with \eqref{eq:sequence} shows that some positive integer $x$ with $a>x>b$ (e.g. $x=b_1=a_2$) appears between $a$ and $b$ in the one-line notation for $w$. This contradicts $(a,b)$ being in $\Refw$. 
\end{proof}

\begin{corollary} \label{coro:simplicial}
	The dual cone $\Dw$ is simplicial if and only if $|\Refw|=\dim \Dw$.
\end{corollary} 

We note that unless the identity $|\Refw|=\dim \Dw$ is satisfied, $\Yw$ is singular by Proposition~\ref{prop:dual}.
The dimension of the cone $\Dw$ is same as the rank of the matrix $M_w$ whose row vectors are $e_{b}-e_a$ for $(a,b) \in \Gw$. Since the nullity of the matrix $M_w$ is the number of connected components of $\Gw$, by the rank theorem, we have that
\begin{equation}\label{eq_dim_Dw_and_rank}
\dim \Dw = |\text{{\rm vertices of} $\Gw$}|-|\text{{\rm connected components of} $\Gw$}|.
\end{equation}
\begin{example} \label{exam:2}
	\begin{enumerate}
		\item Take $w=3152674$ in Example~\ref{example_Rw}(1). The cone $\Dw$ is spanned by 
		the row vectors of the matrix
		\[
		M_{3152674} = \begin{bmatrix}
		1 & 0 & -1 & 0 & 0 & 0 & 0 \\
		0 & 1 & -1 & 0 & 0 & 0 & 0 \\
		0 & 1 & 0 & 0 & -1 & 0 & 0 \\
		0 & 0 & 0 & 1 & -1 & 0 & 0 \\
		0 & 0 & 0 & 1 & 0 & -1 &0 \\
		0 & 0 & 0 & 1 & 0 & 0 & -1
		\end{bmatrix}. 
		\]	
		Then the null space of $M_{3152674}$ is spanned by $(1,1,1,1,1,1,1) \in \R^7$, so 	the nullity of $M_{3152674}$ is $1$.
		Hence the dimension of the cone $\Dw$ is $7-1 = 6$. Indeed, the graph $\G_{3152674}$ is connected (see~Figure~\ref{graph_3152674}).
	Since $|\Refw|$ is also $6$, $\Dw$ is simplicial and hence the merged cone $\Cbw$ in the fan is also simplicial by Proposition~\ref{prop:dual}. The corresponding graph $\Gw$ is a path graph (see~Figure~\ref{graph_3152674}).
		\item Take $w=3715264$ in Example~\ref{example_Rw}(2). Then $|\Refw|=8$ but $\dim \Dw=6$ since the graph~$\G_{3715264}$ is connected (see~Figure~\ref{graph_3715264}). Therefore $\Yw$ is singular. The corresponding graph $\Gw$ is connected and $b_1(\Gw)=2$ where $b_1$ denotes the first Betti number (see~Figure~\ref{graph_3715264}).
		\item Take $w=3412$. Then $|\Refw|=4$ but $\dim \Dw=3$. Therefore $\Yw$ is singular. The corresponding graph $\Gw$ is connected and $b_1(\Gw)=1$ (see~Figure~\ref{graph_3412}).
		\item Take $w=341265$. Then $|\Refw|=5$ but $\dim \Dw=4$. Therefore $\Yw$ is singular. The corresponding graph $\Gw$ has two connected components and $b_1(\Gw)=1$ (see~Figure~\ref{graph_341265}).
	\end{enumerate}
\end{example}
\begin{figure}
	\centering 
	\begin{minipage}[b]{0.5\textwidth}
		\centering
		\begin{tikzpicture}[scale = 0.5]
		
		\draw[->, gray, very thin] (0,0) -- (4.5,0) node[very near end, sloped, below] {position};
		\draw[->, gray, very thin] (0,0) -- (0,4.5) node[very near end, sloped, above] {value};
		\begin{scope}[color=gray!50, thin]
		\foreach \xi in {1,...,4} {\draw (\xi,0) -- (\xi, 4);}%
		\foreach \yi in {1,...,4} {\draw (0,\yi) -- (4,\yi);}
		\end{scope}
		
		\node (3) [circle,draw, fill=white!20, inner sep = 0.25mm] at (1,3) {3};
		\node (4) [circle,draw, fill=white!20, inner sep = 0.25mm] at (2,4) {4};
		\node (1) [circle,draw, fill=white!20, inner sep = 0.25mm] at (3,1) {1};
		\node (2) [circle,draw, fill=white!20, inner sep = 0.25mm] at (4,2) {2};
		
		\draw[->-=.5] (3)--(1);
		\draw[->-=.4] (3)--(2);
		\draw[->-=.4] (4)--(1);
		\draw[->-=.5] (4)--(2);
		
		\end{tikzpicture}
		\caption{Graph $\Gamma_{3412}$.}
		\label{graph_3412}
	\end{minipage}%
	\begin{minipage}[b]{0.5\textwidth}
		\centering
		\begin{tikzpicture}[scale = 0.5]
		\draw[->, gray, very thin] (0,0) -- (6.5,0) node[very near end, sloped, below] {position};
		\draw[->, gray, very thin] (0,0) -- (0,6.5) node[very near end, sloped, above] {value};
		\begin{scope}[color=gray!50, thin]
		\foreach \xi in {1,...,6} {\draw (\xi,0) -- (\xi, 6);}%
		\foreach \yi in {1,...,6} {\draw (0,\yi) -- (6,\yi);}
		\end{scope}
		
		\node (3) [circle,draw, fill=white!20, inner sep = 0.25mm] at (1,3) {3};
		\node (4) [circle,draw, fill=white!20, inner sep = 0.25mm] at (2,4) {4};
		\node (1) [circle,draw, fill=white!20, inner sep = 0.25mm] at (3,1) {1};
		\node (2) [circle,draw, fill=white!20, inner sep = 0.25mm] at (4,2) {2};
		\node (6) [circle,draw, fill=white!20, inner sep = 0.25mm] at (5,6) {6};
		\node (5) [circle,draw, fill=white!20, inner sep = 0.25mm] at (6,5) {5};
		
		\draw[->-=.5] (3)--(1);
		\draw[->-=.4] (3)--(2);
		\draw[->-=.4] (4)--(1);
		\draw[->-=.5] (4)--(2);
		\draw[->-=.5] (6)--(5);
		\end{tikzpicture}
		\caption{Graph $\Gamma_{341265}$.}
		\label{graph_341265}
	\end{minipage}
\end{figure}

Since $|\Refw|$ is nothing but the number of the edges in $\Gw$, the equality~\eqref{eq_dim_Dw_and_rank} proves the following. 

\begin{lemma} \label{lemm:R(w)\Dw}
	We have
	\[ \dim_\C \Yw=\dim \Dw=|\text{{\rm vertices of} $\Gw$}|-|\text{{\rm connected components of} $\Gw$}|.\]
	Moreover
	\[ |\Refw|-\dim \Dw=b_1(\Gamma_w)\ \ge 0,\]
	where $b_1(\Gw)$ denotes the first Betti number of $\Gw$, and hence $\Dw$ is simplicial if and only if $\Gw$ is a forest {\rm (}as an undirected graph{\rm )} by Corollary~\ref{coro:simplicial}. 
\end{lemma}

Remember that our cone $\Dw$ has $\{e_b-e_a\mid (a,b)\in E(\Gw)\}$ as edge vectors (Lemma~\ref{lemm:edge}). They lie in the linear subspace, denoted by $H^{n-1}$, of $\R^n$ with the sum of the coordinates equal to zero. In general, we have the following. 

\begin{lemma}\label{lem_Dw_simplicial_iff_smooth}
	Let $\G$ be a directed graph with vertices in $\{1,\dots,n\}$ and let $D$ be the cone in $\R^n$ with $\{e_b-e_a\mid (a,b)\in E(\G)\}$ as edge vectors, where $E(\G)$ denotes the set of directed edges in $\Gamma$. Then the following are equivalent:
	\begin{enumerate}
		\item $D$ is simplicial.
		\item $\G$ is a forest {\rm (}as an undirected graph{\rm )}.
		\item $D$ is non-singular, i.e., the set $\{e_b-e_a\mid (a,b)\in E(\G)\}$ is a part of a $\Z$-basis of $H^{n-1}\cap \Z^n$.
	\end{enumerate}
	In particular, $\Dw$ is simplicial if and only if $\Dw$ is non-singular.
\end{lemma}

\begin{proof}
	The equivalence of (1) and (2) can be seen by the same argument as above for $\Dw$. It is obvious that (3) implies (1), so it is enough to show that (2) implies (3). We may assume that $\G$ is a tree. To prove (3), we may change the order of coordinates of $\R^n$ and note that changing the order of the coordinates is nothing but relabeling the vertices of $\G$. 
	
	We shall relabel the vertices of $\G$.  For any two vertices $x$ and $y$ of $\G$, there is a unique path connecting them since $\G$ is a tree.  We define the distance $d(x,y)$ to be the number of edges in the path. Choose any vertex of $\G$ and label it as~$1$.  For any positive integer $d$, consider the set
	\[
	A_d:=\{ v\in V(\Gamma)\mid d(v,1)=d\}
	\]
	where $V(\Gamma)$ denotes the set of vertices of $\G$. 
	Then we label the vertices in $A_1$ as $2,3,\dots,|A_1|+1$ and then label the vertices in $A_2$ as $|A_1|+2,|A_1|+3,\dots, |A_1|+|A_2|+1$ and so on. We orient each edge $(a,b)$ of this relabeled graph in such a way that $a>b$ and denote the relabeled directed graph by $\bar{\G}$. Then it is not difficult to see that $\{e_b-e_a\mid (a,b)\in E(\bar{\G})\}$ is a $\Z$-basis of the free abelian group generated by $e_1-e_2,e_2-e_3,\dots,e_{m-1}-e_m$, where $m$ is the number of vertices of $\Gamma$, and hence a part of a $\Z$-basis of $H^{n-1}\cap \Z^n$, proving (3). 
\end{proof}

\begin{example}
	Take $w = 3152674$ in~Example~\ref{example_Rw}(1). Then the graph $\Gw$ is in Figure~\ref{graph_3152674}.
	We can reorder the vertices and change directions on edges as 
	follows.
	
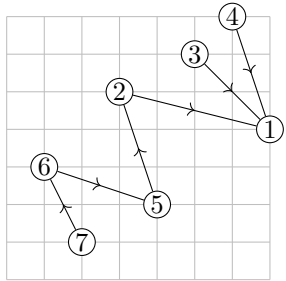
\begin{center}
		\begin{tikzpicture}[scale = 0.5]
		\begin{scope}[color=gray!50, thin]
		\foreach \xi in {0,1,2,3,4,5,6,7} {\draw (\xi,0) -- (\xi, 7);}%
		\foreach \yi in {0,1,2,3,4,5,6,7} {\draw (0,\yi) -- (7,\yi);}
		\end{scope}
		
		\node (3) [circle,draw, fill=white!20, inner sep = 0.25mm] at (1,3) {6};
		\node (1) [circle,draw, fill=white!20, inner sep = 0.25mm] at (2,1) {7};
		\node (5) [circle,draw, fill=white!20, inner sep = 0.25mm] at (3,5) {2};
		\node (2) [circle,draw, fill=white!20, inner sep = 0.25mm] at (4,2) {5};
		\node (6) [circle,draw, fill=white!20, inner sep = 0.25mm] at (5,6) {3};
		\node (7) [circle,draw, fill=white!20, inner sep = 0.25mm] at (6,7) {4};
		\node (4) [circle,draw, fill=white!20, inner sep = 0.25mm] at (7,4) {1};
		
		\draw[->-=.5] (1)--(3);
		\draw[->-=.5] (3)--(2);
		\draw[->-=.5] (2)--(5);
		\draw[->-=.5] (5)--(4);
		\draw[->-=.5] (6)--(4);
		\draw[->-=.5] (7)--(4);
		\end{tikzpicture}
	\end{center}
	Then the corresponding edge vectors are 
	\[
	\{e_1 - e_2, e_1 - e_3, e_1 - e_4, e_2- e_5, e_5 - e_6, e_6 - e_7\},
	\]
	which is a part of $\mathbb Z$-basis of $H^{6} \cap \mathbb{Z}^7$.
\end{example}

Combining Corollary~\ref{coro:simplicial}, Lemmas~\ref{lemm:R(w)\Dw} and~\ref{lem_Dw_simplicial_iff_smooth}, we obtain 
\begin{theorem}\label{thm_Dw_simplicial_iff_Gw_forest}
	The generic torus orbit closure $\Yw$ is smooth at the fixed point $wB$ if and only if $\Gw$ is a forest.
\end{theorem}

To our surprise, the graph $\Gw$ has been studied and the following is proven.

\begin{proposition}[\cite{bm-bu07}] \label{prop:avoidence}
	The graph $\Gw$ is a forest if and only if $w$ avoids the patterns $4231$ and $45\bar{3}12$.
\end{proposition}

Here, a permutation $w$ \emph{avoids the pattern $4231$} if one cannot find indices $i<j<k<\ell$ such that $w(\ell)<w(j)<w(k)<w(i)$. Similarly $w$ \emph{avoids the pattern $45\bar{3}12$} if every occurrence of the pattern $4512$ is a subsequence of an occurrence of $45312$. 
In \ref{section_acyclicity_and_avoidance}, we will provide a proof different from the proof in~\cite{bm-bu07}.
Actually, in~\cite{bm-bu07}, they provide one more equivalent condition on the statement of Proposition~\ref{prop:avoidence}. And they used this condition to prove the equivalence in the previous proposition. But we provide a direct proof of Proposition~\ref{prop:avoidence}.

\begin{remark} 
	In \cite{bm-bu07}, they associate a graph $G_\pi$ to a permutation $\pi$ and prove that $G_\pi$ is a forest if and only if $\pi$ avoids the patterns $1324$ and $21\bar{3}54$.  In fact, our $\Gw$ is their $G_{w_0w}$, so we obtain the statement in Proposition~\ref{prop:avoidence}.
	%
\end{remark}

We restate Theorem~\ref{thm_Dw_simplicial_iff_Gw_forest} in terms of pattern
avoidance using Proposition~\ref{prop:avoidence}.

\begin{theorem}
	The generic torus orbit closure  $\Yw$ is smooth at the fixed point $w B$ if and only if $w$ avoids the patterns $4231$ and $45\bar{3}12$.
\end{theorem}

\begin{remark}
	According to \cite{bm-bu07} and \cite{wo-yo06}, $\Gw$ is a forest if and only if our Schubert variety $\Xw$ $(=\overline{BwB/B})$ is factorial. 
	(Note. $X_w$ in \cite{wo-yo06} is different from our $\Xw$, indeed their $\Xw$ is $\Omega_w$ in Fulton's book \cite{Fulton97Young}, which is the closure of the \emph{opposite} Schubert cell  of codimension~$\ell(w)$.) 
\end{remark}

\section{Smoothness of $\Yw$ at other fixed points} \label{sect:6}

In the previous sections, we identified the cone dual to the maximal cone $\Cbw=\Cbw(w)$ and associated the graph $\Gw$.  In this section, we will do the same task to the other maximal cones $\Cbwu$ $(u\le w)$  generalizing the previous case, namely we identify the cone dual to $\Cbwu$ and associate a graph, denoted by $\Gwu$, to each $u\le w$. It seems that if the graph $\Gw$ is a forest, then so is $\Gwu$ for any $u\le w$.  This means that $\Yw$ is smooth if it is smooth at the fixed point $wB$.  We propose this and a slightly more general statement as a conjecture at the end of this section.

We begin with the generalization of $\Refw$ introduced in Definition~\ref{defi:E(w)}.

\begin{definition} \label{defi:tildeE_u(w)}
	For $u\le w$, we define
	\[
	\wideRefwu:=\{ (u(i),u(j))\mid \ 1\le i<j\le n,\ t_{u(i),u(j)}u\le w,\ {|\ell(u)-\ell(t_{u(i),u(j)}u)|=1}\}
	\] 
	where $t_{a,b}$ denotes the transposition of $a$ and $b$ {and $\ell(v)$ denotes the length of a permutation $v$ as before}.  
\end{definition}

\begin{remark} 
	\begin{enumerate}
			\item $\widetilde\Ref_w(w)=\Refw$ because the condition $\ell(w)-\ell(t_{w(i),w(j)}w)=1$ in Definition~\ref{defi:E(w)} is equivalent to $$t_{w(i),w(j)}w\le w \quad \text{and}\quad |\ell(w)-\ell(t_{w(i),w(j)}w)|=1.$$
			\item The condition $|\ell(u)-\ell(t_{u(i),u(j)}u)|=1$ in Definition~\ref{defi:tildeE_u(w)} is equivalent to 
			\[
			\begin{split}
			&\text{ $u(k)\notin [u(j),u(i)]$\quad  when $u(i)>u(j)$},\\
			&\text{ $u(k)\notin [u(i),u(j)]$\quad when $u(i)<u(j)$},
			\end{split}
			\]
			for $i<\forall k<j$.
			\item	The sets $\Refw$ and $\wideRefwu$ have been studied by Tsukerman and Williams in~\cite{ts-wi15}.  Using the notations $\overline{T}(u,[id,w]) $ and $ \underline{T}(u, [id,w])$ in~\cite{ts-wi15}, we have that 
			\begin{equation}\label{eq_relation_Ewu_T}
			\wideRefwu = \{(u(i), u(j)) \mid  (i,j) \in \overline{T}(u,[id,w]) \cup \underline{T}(u, [id,w])\}.
			\end{equation}
		Here, 
	\[
	\begin{split}
	&\overline{T}(u, [id,w]) := \{ (i,j) \mid 1 \leq i < j \leq n, \ t_{u(i),u(j)} u < u, \ |\ell(u)-\ell(t_{u(i),u(j)}u)|=1 \}, \\
	&\underline{T}(u, [id,w]) := \{ (i,j) \mid 1 \leq i < j \leq n, \ u < t_{u(i),u(j)} u \leq w, \ |\ell(u)-\ell(t_{u(i),u(j)}u)|=1 \}.
	\end{split}
\]
	\end{enumerate}
\end{remark}

\begin{example} \label{exam:tildeRwu}
	Take $w=3412$. 
	\begin{enumerate}
		\item If $u=2143$, then 
		$ \wideRefwu=\{ (1,4), (2,3), (2,1), (4,3)\}.$
		\item If $u=2413$, then 
		$ \wideRefwu=\{ (2,3), (2,1), (4,1), (4,3)\}. $
		\item If $u=1432$, then 
		$ \wideRefwu=\{ (1,3), (4,3), (3,2)\}.$
	\end{enumerate}
\end{example}

We define $(a,b)+(b,c)=(a,c)$. Then, in (1) in the example above, we have 
\[
(2,3)=(2,1)+(1,4)+(4,3).
\]
We say that an element of $\wideRefwu$ is \emph{decomposable} if it is sum of some other elements in $\wideRefwu$, and \emph{indecomposable} otherwise. 

\begin{definition}
	We define
	\[
	\Refwu:=\{\text{indecomposable elements in $\wideRefwu\}$}
	\]
	and 
	\[
	\Dwu:=\text{ the cone in $\V$ spanned by $\{ e_b-e_a\mid (a,b)\in\Refwu\}$}.
	\]
\end{definition}

\begin{remark} \label{rema:E_u(w)}
	\begin{enumerate}
		\item We have that
			\[
			\Dwu = \text{ the cone in $\V$ spanned by $\{e_b - e_a \mid (a,b) \in \wideRefwu \}$}
			\]
			since the relation $(a,b) = (a,c) + (c,b)$ implies that $e_a - e_b = (e_a - e_c) + (e_c - e_b)$.
		\item For $u \leq w$, Let $G^w_u$ be a directed graph whose nodes are $[n]$ and directed edges are given by
			\[
			\begin{split}
			(i,j) \in G^w_u \quad \text{ if } (i,j) \in \overline{T}(u, [id,w]), \\
			(j,i) \in G^w_u \quad \text{ if }(i,j) \in \underline{T}(u,[id,w]).
			\end{split}
			\]
			Then it is proved in~\cite[Theorem~4.19]{ts-wi15} that this graph is a directed acyclic graph.
			Hence it has a unique \textit{transitive reduction}, that is, there is a  unique subgraph of $G^w_u$ obtained by removing each directed edge $(i,j)$ if there is a directed path from $i$ to $j$ (see~\cite{AGU72_transitive} for more details). 
			Since the permutation $u \in \mathfrak{S}_n$ maps the set $\overline{T}(u, [id,w]) \cup \underline{T}(u,[id,w])$ bijectively onto $\wideRefwu$ (see~\eqref{eq_relation_Ewu_T}), an element $(a,b)$ in $\wideRefwu$ is decomposable if and only if the edge $(u^{-1}(a), u^{-1}(b))$ in $G^w_u$ disappears in its transitive reduction.
			Because of the uniqueness of the transitive reduction of a directed acyclic graph, 			
			the set $E_w(u)$ of indecomposable elements is well-defined. 
		\item $\Ref_w(w)=\widetilde\Ref_w(w)=\Refw$, so $D_w(w)=\Dw$,
		\item $\Refwou=\{(u(i),u(i+1))\mid i=1,2,\dots,n-1\}$.
		\item When $u=id$ (the identity permutation), $E_w(id)=\{(a,a+1)\mid t_{a,a+1}id\le w\}$.  \label{Ew_id}
	\end{enumerate}
\end{remark}

\begin{example} \label{exam:Rwu}
	For (1) in Example~\ref{exam:tildeRwu}, 
	$ \Refwu=\{ (1,4), (2,1), (4,3)\},$ 
	and for (2) and (3), $\Refwu=\wideRefwu$. One can check that $\Cbwu=\Dwu^\vee$ in these cases. 
\end{example}

The purpose of this section is to prove the following proposition.

\begin{proposition}\label{prop_dual_Cuw}
	$\Cbwu=\Dwu^\vee$ for any $u\le w$.
\end{proposition}

As a first step, we prove the following lemma.

\begin{lemma} \label{lemm:w0}
	$\Cu=\Dwou^\vee$ for any $u\in \frak{S}_n$.
\end{lemma} 

\begin{proof}
	Let $(a,b)$ be an arbitrary element in $\Ref_{w_0}(u)$. Then $a$ appears ahead of $b$ in the one-line notation for $u$. Therefore $e_b-e_a$ is non-negative on the cone $\Cu$, which means that $\Cu\subset \Dwou^\vee$. 
	Suppose $\Cu\subsetneq \Dwou^\vee$ (and we deduce a contradiction). Then there is some $v(\not=u)\in \frak{S}_n$ such that the relative interior of the cone $\Cv$ intersects with $\Dwou^{\vee}$.
	Since $v\not=u$, there is a pair $(a,b)$ such that $a$ appears ahead of $b$ in the one-line notation for $u$ while $b$ appears ahead of $a$ in the one-line notation for $v$. Suppose $a>b$. Then there is a sequence of pairs $(a_{i-1},a_{i})\in \Refwou$ $(i=1,\dots,k)$ such that 
	\begin{equation} \label{eq:inequality}
	a=a_0>a_1>\cdots >a_{k}=b.
	\end{equation} 
	Let $x \in \Int(C(v)) \cap \Dwou^{\vee}$.
	Since $(a_{i-1},a_i)\in \Refwou$, 
	\[
	\langle e_{a_i} - e_{a_{i-1}}, x \rangle \geq 0.
	\]
	This means that in the one-line notation for $v$, $a_{i-1}$ appears ahead of $a_i$ for any $i=1,\dots,k$ so that $a$ appears ahead of $b$, which is a contradiction. The same argument works when $a<b$ if the inequalities in \eqref{eq:inequality} are reversed. 
\end{proof}

\begin{lemma} \label{lemm:Cvcontained}
	Let $u\le w$. If $v'=u$, then $\Cv\subset \Dwu^\vee$. Therefore $\Cbwu\subset \Dwu^\vee$ by Corollary~\ref{coro:maximal_cone}.
\end{lemma}

\begin{proof}
	Since $v'=u$, we have $v\ge_R u$ by Lemma~\ref{lemm:vprime}. Therefore $v$ is of the form
	\[
	v=us_{i_1}\cdots s_{i_k} \quad\text{with $\ell(v)=\ell(u)+k$}.
	\]
	We shall prove the lemma by induction on $k$.
	When $k=0$, that is, $v=u$, $\Cv=\Cu$. By Lemma~\ref{lemm:w0}, $\Cu=\Dwou^\vee$. Since $\Dwou\supset \Dwu$, we obtain $\Cu\subset \Dwu^\vee$ by taking their dual. This proves the lemma when $k=0$. 
	
	Suppose $k\ge 1$ and the lemma holds for $v$ with $\ell(v)\le \ell(u)+k-1$. We set 
	\[
	\bar{v}=us_{i_1}\cdots s_{i_{k-1}}\qquad \text{and}\qquad i_k=p.
	\]
	Then $v=\bar{v}s_p$, i.e.
	\begin{equation} \label{eq:vbarv}
		\arraycolsep=1.4pt
	\begin{array}{rcccccccc}
	\bar{v}&=v(1)&\cdots &v(p-1)&v(p+1)&v(p)&v(p+2)&\cdots& v(n),\\
	{v}&=v(1)&\cdots &v(p-1)&v(p)&v(p+1)&v(p+2)&\cdots &v(n),
	\end{array}
	\end{equation}
	and $\Cbarv\subset \Dwu^\vee$ by induction assumption (note that $\bar{v}'=u$ since $\bar{v}\in [u,v]_R$ and $v'=u$). 
	Let $(a,b)$ be an arbitrary element of $\Refwu$. Since $e_b-e_a$ is non-negative on $\Cbarv$, $a$ appears ahead of $b$ in the one-line notation for $\bar{v}$. 
	
	If $\{v(p),v(p+1)\}\not=\{a,b\}$, then the order of $a$ and $b$ in the one-line notation for $v$ is same as that for $\bar{v}$. Therefore, $e_b-e_a$ is non-negative on $\Cv$ as well. Thus it suffices to show that the case where $\{v(p),v(p+1)\}=\{a,b\}$ does not occur. 
	Suppose $\{v(p),v(p+1)\}=\{a,b\}$. Then we have
	\begin{equation} \label{eq:avpvpb}
	a=v(p+1)<v(p)=b
	\end{equation}
	because $a$ appears ahead of $b$ in the one-line notation for $\bar{v}$ and $v=\bar{v}s_p$ with $\ell(v)=\ell(\bar{v})+1$. On the other hand, since $(a,b)\in\Refwu$, $a$ appears ahead of $b$ in the one-line notation for $u$; so $a=u(s)$ and $b=u(t)$ with some $s<t$, i.e., 
	\begin{equation*} \label{eq:tabu}
		\arraycolsep=1.4pt
	\begin{array}{rccccccccccc}
	u&=u(1)&\cdots &u(s-1)& a& u(s+1)&\cdots &u(t-1)& b& u(t+1)&\cdots &u(n),\\
	t_{a,b}u&=u(1)&\cdots &u(s-1)& b& u(s+1)&\cdots &u(t-1)& a& u(t+1)&\cdots& u(n).
	\end{array}
	\end{equation*}
	Since $(a,b)\in\Refwu$, we have $t_{a,b}u\le w$ by definition. This means that 
	\begin{equation} \label{eq:usb}
	\{u(1),\dots,u(s-1),b\}\uparrow\le w^{(s)}.
	\end{equation}
	Now remember that $\bar{v}'=u$. This together with \eqref{eq:vbarv}, \eqref{eq:avpvpb} and \eqref{eq:usb} shows that $v'(s)=b$ and hence $v'\not=u$, a contradiction. 
\end{proof}

\begin{lemma} \label{lemm:6-1}
	Let $x\le w$, $y\le w$ and $x\not=y$. If the cones $\Cwx$ and $\Cwy$ share a common facet, then they are related by a transposition, i.e., $x=t_{a,b}y$ for some transposition $t_{a,b}$.
\end{lemma}

\begin{proof}
	$x$ and $y$ are fixed points in the flag variety $G/B$ under the torus action and the assumption implies that they are in $\C P^1$ fixed pointwise under the codimension one subtorus corresponding to the facet $\Cwx\cap \Cwy$. In other words, $x$ and $y$ are vertices of the GKM graph of $G/B$ and they are joined by an edge. 
	Therefore they are related by a transposition.
	Recall from \cite{care94} and \cite[Proposition~2.1]{tymo08} that the GKM graph of $G/B$ is the labeled directed graph whose vertices are permutations $w \in \mathfrak{S}_n$ and edges are given by each pair of permutations $(w, t_{a,b}w)$ for some transposition $t_{a,b}$. (See~\cite{GKM98,GZ01} for the definition of GKM manifolds and GKM graphs.)
\end{proof}
\begin{remark}
	Lemma~\ref{lemm:6-1} can be proven using the fact that the fan of $Y_w$ is the normal fan of $Q_{id, w^{-1}}$ (see~Remark~\ref{rmk_moment_map}). By~\cite[Theorem~4.1]{ts-wi15},  every face of a Bruhat interval polytope is a Bruhat interval polytope. Hence an edge of $Q_{id, w^{-1}}$ is $Q_{x,y}$ for certain permutations $x$ and $y$. Since $Q_{x,y}$ is one-dimensional, the Bruhat interval $[x,y]$ consists of two elements $\{x,y\}$, so that $x = t_{a,b}y$, which proves Lemma~\ref{lemm:6-1}.
\end{remark}

We prepare one more lemma.

\begin{lemma} \label{lemm:6-2}
	Let $u\le w$. If $t_{u(j),u(k)}u\le w$ with $j<k$, then $e_{u(k)}-e_{u(j)}\in \Dwu$. 
\end{lemma}

\begin{proof}
	It suffices to show that 
	\begin{quote}
		$(*)$\qquad $(u(j),u(k))$ is a finite sum of elements in $\wideRefwu$. 
	\end{quote}
	If there is no $j<m<k$ such that $u(m)$ is in between $u(j)$ and $u(k)$ (we say that $u(j)$ and $u(k)$ are \emph{adjacent} in this case), then $(u(j),u(k))\in \wideRefwu$ by definition of $\wideRefwu$ and hence $e_{u(k)}-e_{u(j)}\in \Dwu$ by definition of $\Dwu$. Thus we may assume that $u(j)$ and $u(k)$ are not adjacent. 
	We consider two cases. 
	
	Case 1. The case where $u(j)>u(k)$. In this case, there is a sequence $j<m_1<m_2<\dots<m_p<k$ such that 
	\[
	\begin{split}
	&u(j)>u(m_1)>u(m_2)>\dots>u(m_p)>u(k)\quad\text{and}\\
	&\text{$u(m_\ell)$ and $u(m_{\ell+1})$ are adjacent for $\ell=0,1,\dots,p$, where $m_0=j, m_{p+1}=k$.}
	\end{split}
	\]
	Since $(u(m_\ell),u(m_{\ell+1}))$ is an inversion of $u$ and $u\le w$, we have $t_{u(m_\ell),u(m_{\ell+1})}u\le w$ and hence $(u(m_\ell),u(m_{\ell+1}))\in \wideRefwu$ for $\ell=0,1,\dots,p$. This shows that 
	\begin{equation} \label{eq:6-0}
	(u(j),u(k))=\sum_{\ell=0}^p (u(m_\ell),u(m_{\ell+1}))
	\end{equation}
	proving the assertion $(*)$. 
	
	Case 2. The case where $u(j)<u(k)$. Similarly to Case 1 above, there is a sequence $j<m_1<m_2<\dots<m_p<k$ such that 
	\[
	\begin{split}
	&u(j)<u(m_1)<u(m_2)<\dots<u(m_p)<u(k)\quad\text{and}\\
	&\text{$u(m_\ell)$ and $u(m_{\ell+1})$ are adjacent for $\ell=0,1,\dots,p$, where $m_0=j, m_{p+1}=k$.}
	\end{split}
	\]
	In this case, since $(u(m_\ell),u(m_{\ell+1}))$ is not an inversion of $u$, it is not immediate that $t_{u(m_\ell),u(m_{\ell+1})}u\le w$. However, $t_{u(j),u(k)}u\le w$ by assumption and this will imply $t_{u(m_\ell),u(m_{\ell+1})}u\le w$. Indeed, setting $u(m_\ell)=a_\ell$ for simplicity, one can see that 
	\begin{equation} \label{eq:6-1}
	\begin{split}
	&t_{a_0,a_1}u=t_{a_{p+1},a_1}t_{a_1,a_0}(t_{a_0,a_{p+1}}u),\\
	&t_{a_\ell,a_{\ell+1}}u=t_{a_{p+1},a_\ell}
	t_{a_{\ell+1},a_0}t_{a_\ell,a_0}t_{a_{p+1},a_{\ell+1}}(t_{a_0,a_{p+1}}u)\quad\text{for $\ell=1,\dots,p$}.
	\end{split}
	\end{equation}
	Here $t_{a_0,a_{p+1}}u=t_{u(j),u(k)}u\le w$ by assumption and all the transpositions $t_{a,b}$ in~\eqref{eq:6-1} (except $t_{a_0,a_{p+1}}$) have $a>b$ and $(a,b)$ is an inversion of the permutation, say $v$, which $t_{a,b}$ is applied to; so if $v\le w$, then $t_{a,b}v\le w$. Therefore \eqref{eq:6-1} shows that $t_{a_\ell,a_{\ell+1}}u\le w$ for $\ell=0,1,\dots,p$ and hence $(u(m_\ell),u(m_{\ell+1}))=(a_\ell,a_{\ell+1})\in \wideRefwu$, proving the assertion $(*)$ by \eqref{eq:6-0}. 
\end{proof}

Now we are in a position to prove Proposition~\ref{prop_dual_Cuw}.

\begin{proof}[Proof of Proposition~\ref{prop_dual_Cuw}]
	We know $\Cbwu\subset \Dwu^\vee$ by Lemma~\ref{lemm:Cvcontained}. Suppose that $\Cbwu\subsetneq \Dwu^\vee$. Then there exist $v\in\frak{S}_n$ and a simple reflection $s_i$ such that 
	\begin{equation} \label{eq:6-2}
	\Cv\subset \Cbwu,\quad \Cvsi\nsubseteq \Cbwu,\quad \Int(\Cvsi) \cap \Dwu^{\vee} \neq \emptyset.
	\end{equation}
	
	\medskip
	\noindent
	{\bf Claim.} $e_{v(i+1)}-e_{v(i)}\in \Dwu$. 
	\medskip
	
	We assume the claim and complete the proof. 
	Let $y \in \Int(\Cvsi) \cap \Dwu^{\vee} \neq \emptyset$. Then
	the claim says that $\langle e_{v(i+1)}-e_{v(i)}, y \rangle \geq 0$. On the other hand, in the one-line notation of $vs_i$, $v(i+1)$ appears on the left of $v(i)$ and this means that $\langle e_{v(i+1)}-e_{v(i)}, y \rangle \leq 0$. Therefore $\langle e_{v(i+1)}-e_{v(i)}, y \rangle = 0$ but this contradicts $y \in \Int(\Cvsi)$. 
	Thus, it suffices to prove the claim above. 
	
	By \eqref{eq:6-2}, $\Cvsi\subset \Cwx$ for some $x(\not=u)\le w$. The intersection $\Cv\cap \Cvsi$ is a facet of $\Cv$ and $\Cvsi$, and $\Cbwu\cap \Cwx$ contains $\Cv\cap \Cvsi$, so $\Cbwu$ and $\Cwx$ are adjacent. Therefore $x=t_{a,b}u$ for some transposition $t_{a,b}$ by Lemma~\ref{lemm:6-1}. Since $\Cvsi\subset \Cwx$ and $\Cv\subset \Cbwu$, it follows that 
	\begin{equation} \label{eq:6-3}
	(vs_i)'=x=t_{a,b}u\quad\text{and}\quad v'=u.
	\end{equation}
	We define $j$ and $k$ by 
	\begin{equation} \label{eq:6-4}
	v(i)=u(j),\quad v(i+1)=u(k).
	\end{equation}
	We consider two cases. 
	
	Case 1. The case where $v(i)<v(i+1)$. Since $v'=u$, we have $j<k$ in this case. This together with $(vs_i)'\not=u$ implies that 
	\begin{equation} \label{eq:6-5}
	\arraycolsep=1.4pt
	\begin{array}{rccccccccc}
	v'=u&=u(1)&\cdots &u(j-1)&u(j)&\cdots, \\
	(vs_i)'&=u(1)&\cdots& u(j-1)&u(k)&\cdots.
	\end{array}
	\end{equation}
	On the other hand, we know $(vs_i)'=t_{a,b}u$ by \eqref{eq:6-3}. This together with \eqref{eq:6-5} implies that 
	\[
	(vs_i)'=t_{u(j),u(k)}u.
	\]
	Since $j<k$ and $(vs_i)'\le w$, the above together with Lemma~\ref{lemm:6-2} shows that $e_{u(k)}-e_{u(j)}\in \Dwu$. Then the claim follows from \eqref{eq:6-4}. 
	
	Case 2. The case where $v(i)>v(i+1)$. We claim $j<k$ in this case, too. Indeed, since $v(i+1)$ appears on the left of $v(i)$ in the one-line notation of $vs_i$ and $v(i+1)<v(i)$, one sees that $v(i+1)(=u(k))$ appears ahead of $v(i)(=u(j))$ in the one line notation of $(vs_i)'$. Therefore, if $k<j$, then $(vs_i)'$ coincides with $u$ which contradicts the assumption $(vs_i)'\not=u$. Therefore $j<k$. Then \eqref{eq:6-5} holds in this case too and the claim follows from the same argument as in Case~1. 
	This completes the proof of the claim and the proposition. 
\end{proof}

We associate a graph $\Gwu$ to $\Refwu$ as before. 
Note that $\Gamma_w(id)$ is the disjoint union of path graphs (in particular, a forest) by Remark~\ref{rema:E_u(w)}(\ref{Ew_id}).  
Then by Lemma~\ref{lem_Dw_simplicial_iff_smooth}, we obtain

\begin{corollary}\label{cor_Duw_smooth_iff_Guw_forest}
	The dual cone $\Dwu$ is smooth if and only if $\Gwu$ is a forest. 
	Indeed, $\Yw$ is smooth at the fixed point $uB$ if and only if $\Gwu$ is a forest.
	In particular, $Y_w$ is smooth at the fixed point $id B$ for any $w$.   
\end{corollary}
\begin{remark}
	Since $Y_w$ has at least one smooth $T$-fixed point $id B$, it follows from~\cite[Proposition 7.2 and Remark 7.3]{CarrellKuttler03} that $Y_w$ is smooth if and only if the Bruhat graph of $Y_w$ is regular, where regular means that there are exactly $\dim_\C Y_w$ $T$-stable curves at each $T$-fixed point in $Y_w$. 
\end{remark}

Using the fact that $id B$ is a smooth point in $Y_w$, we can give an alternative proof to the following result which is indeed proved for any Lie type in \cite{Karu13Schubert}.    

\begin{corollary}[\cite{Karu13Schubert}] \label{coro:toric_Schubert} 
	A Schubert variety $X_w$ is a toric variety with respect to the $T$-action induced by the left multiplication if and only if $w$ is a product of distinct simple reflections. 
\end{corollary}

\begin{proof}
	We first claim that the following are equivalent.
		\begin{enumerate}
			\item The Schubert variety $X_w$ is a toric variety with respect to the $T$-action induced by the left multiplication.
			\item $\dim_{\C}X_w = \dim_{\C}Y_w$.
			\item $X_w = Y_w$.
		\end{enumerate}
It is obvious that the third statement implies the first statement. Moreover, 
if $\dim_{\C}X_w = \dim_{\C}Y_w$, then $X_w = Y_w$ since $X_w$ is irreducible and $X_w \supset Y_w$, so that the second statement implies the third statement. Finally, if $X_w$ is a toric variety with respect to the $T$-action induced by the left multiplication, then the dimension of the moment map image $\mu(X_w)$ equals to its complex dimension $\dim_{\mathbb{C}}X_w$. Since $\mu(Y_w) = \mu(X_w)$ by the Atiyah's convexity theorem~\cite[Theorem~2]{Atiyah82} and $Y_w$ is a toric variety, we have that 
\[
\dim_{\C} X_w = \dim_{\mathbb{R}}(\mu(X_w)) = \dim_{\mathbb{R}}(\mu(Y_w)) = \dim_{\C}Y_w.
\] 
Hence the first statement implies the second statement, and the claim is proved.
	
By the above claim, it is enough to show that $\dim_{\C} X_w = \dim_{\C} Y_w$ if and only if $w$ is a product of simple reflections to prove the corollary.
	As is well-known $\dim_\C X_w=\ell(w)$.  On the other hand, $\dim_\C Y_w$ is equal to the dimension of any  maximal cone in the fan of $Y_w$.  Since the fixed point $id B$ is smooth in $Y_w$, the maximal cone corresponding to $id B$ is simplicial and hence its dimension is equal to $|E_w(id)|$ by Proposition~\ref{prop_dual_Cuw}. Here 
	\[
	E_w(id)=\{ v\in \mathfrak{S}_n\mid \ell(v)=1,\, v\le w\}
	\]
	(see Remark~\ref{rema:E_u(w)}(\ref{Ew_id})) and the condition $\ell(v)=1$ and $v\le w$ is equivalent to $v$ being a simple reflection appearing in a reduced expression of $w$.  This implies the desired result since $\ell(w)$ is the number of simple reflections appearing in a reduced expression of $w$.     
\end{proof}

\begin{remark}
	A toric Schubert variety $X_w$ is a Bott manifold (the top space of a Bott tower in  \cite{GK94Bott}).  Indeed, since $w$ is a product of simple reflections, the Bruhat graph of $X_w$ is combinatorially same as the $1$-skeleton of a cube. This means that the underlying simplicial complex of the fan of $X_w$ is the boundary complex of a cross-polytope.  Therefore $X_w$ is a Bott manifold  (see~\cite[Corollary~3.5]{MP08semifree}).
\end{remark}

We propose the following conjecture.

\begin{conjecture}\label{conj_smoothness_of_Yw}
	The graph $\Gwu$ is a forest for any $u \le w$ if $\Gw$ is a forest.
	{\rm (}This is equivalent to saying that the generic torus orbit closure $\Yw$ is smooth if it is smooth at the fixed point $w B$.{\rm )} 
\end{conjecture}

Here is an example which supports the conjecture.

\begin{example}
	Take $w = 32154$. Then $\Refw=\{(3,2),(2,1),(5,4)\}$. Using simple transpositions $s_i$ 
	for $1 \leq i \leq 4$, 
	we have that $w = s_1s_2s_1s_4$. 
	There are twelve elements $u \in \frak{S}_5$ such that
	$u \leq w$, and we have Table~\ref{table2} for $\Refwu$. We can see that each $\Gwu$ is a forest for $u \leq w$. 
	{
		\begin{table}[H]
			\begin{tabular}{|c|c||c|c|} 
				\hline
				$u$ & $\Refwu$ & $u$ & $\Refwu$ \cr \hline
				$s_1s_2s_1s_4 = 32154$ & $(3,2),(2,1),(5,4)$ &
				$s_1s_2s_1= 32145$ & $(3,2),(2,1),(4,5)$ \cr \hline
				$s_4s_2s_1 = 31254$ & $(3,1),(1,2),(5,4)$ &
				$s_4s_1s_2= 23154$ & $(2,3),(3,1),(5,4)$ \cr \hline
				$s_2s_1 = 31245$ & $(3,1),(1,2),(4,5)$ &
				$s_1s_2= 23145$ & $(2,3),(3,1),(4,5)$ \cr \hline
				$s_4s_1 = 21354$ & $(2,1),(1,3),(5,4)$ &
				$s_4s_2= 13254$ & $(1,3),(3,2),(5,4)$ \cr \hline
				$s_1 = 21345$ & $(2,1),(1,3),(4,5)$ &
				$s_2= 13245$ & $(1,3),(3,2),(4,5)$ \cr \hline
				$s_4 = 12354$ & $(1,2),(2,3),(5,4)$ &
				$e= 12345$ & $(1,2),(2,3),(4,5)$ \cr \hline
			\end{tabular} \vspace{1.25ex}
			\caption{$\Refwu$ for $u \leq w$ when $w = 32154$.}
			\label{table2}
	\end{table}}
\end{example}

One can check Conjecture~\ref{conj_smoothness_of_Yw} for $n\le 5$ with the aid of computers.  Indeed, the computer calculation  suggests that $b_1(\Gamma_w(u))\le b_1(\Gamma_w)$ for any $u\le w$, where $b_1(\Gamma)$ denotes the first Betti number of a graph $\Gamma$.  More strongly, it suggests that $b_1(\Gamma_w(u))\le b_1(\Gamma_w(t_{a,b}u))$ for $(a,b)\in E_w(u)$ with $a<b$.  The latter implies that any two non-simple vertices in the Bruhat interval polytope $Q_{id,w^{-1}}$ (see Remark~\ref{rmk_moment_map} for the definition of $Q_{id,w^{-1}}$) would be joined by edges with non-simple vertices as endpoints (see Figures~\ref{figure_Bruhat_polytope_4231} and~\ref{figure_Bruhat_polytope_3412} in \ref{section_retraction_sequence_poincare_polynomial}).  

\section{Poincar\'e polynomial of $\Yw$} \label{sect:7}

The Eulerian number $A(n,k)$ is the number of permutations in $\mathfrak{S}_n$ with $k$ ascents and the Eulerian polynomial $A_n(t)$ is the generating function of the Eulerian numbers, i.e., $A_n(t)=\sum_{k=0}^{n-1}A(n,k)t^k$.  As is well-known, the Poincar\'e polynomial of the permutohedral variety $\Ywo$ agrees with $A_n(t^2)$ (see, for example, \cite{Klyachko85Orbits}, \cite{Klyachko95toric}, and \cite[\S 4.1]{PRW08_faces}).  
In this section, we consider a generalization of Eulerian numbers, introduce a polynomial $A_w(t)$ for each $w\in \mathfrak{S}_n$, and show that the Poincar\'e polynomial of $\Yw$ is given by $A_w(t^2)$ when $\Yw$ is smooth.

We set
\[
\begin{split}
&\awu:=\#\{ (u(i),u(j))\in \Refwu\mid u(i)<u(j)\}\quad\text{for $u\le w$},\\
&A_w(n,k):=\#\{ u\le w\mid \awu=k\}\quad \text{for $k\ge 0$},
\end{split}
\]
and define
\[
A_w(t):=\sum_{u\le w}t^{\awu}=\sum_{k\ge 0}A_w(n,k)t^k.
\]

\begin{example} \label{exam:Aw}
	Using Figure~\ref{fig_Bruhat_order} and Table~\ref{table_Refwu_and_dwu}, one can find 
	\[
	\begin{split}
	&{A}_w(t)=t^3+11t^2+7t+1 \quad \text{ when }w=4231, \\
	&{A}_w(t)=t^3+7t^2+5t+1  \quad \text{ when }w=3412. 
	\end{split}
	\]
	In these cases, $A_w(t)$ are not palindromic and $\Yw$ are not smooth.  On the other hand, from Table~\ref{table2} one can easily find 
	\[
	\text{$A_w(t)=t^3+5t^2+5t+1$ \quad when $w=32154$.}
	\]  
	In this case, $A_w(t)$ is palindromic and $\Yw$ is smooth. 
	\begin{table}
		\begin{subtable}[b]{.5\textwidth}
			\small
			\centering
			\begin{tabular}[h]{|c|c|c|}
				\hline
				$u$ & $\Refwu$ & $\awu$ \\
				\hline 
				$4231$ & $(4,2), (4,3), (2,1), (3,1)$ & $0$ \\
				\hline
				$4132$ & $(4,1), (4,3), \textcolor{red}{(1,2)}, (3,2)$ & $1$\\
				\hline
				$4213$ & $(4,2), (2,1), \textcolor{red}{(1,3)}$ & $1$ \\
				\hline
				$2431$ & $\textcolor{red}{(2,4)}, (4,3), (3,1)$ & $1$ \\
				\hline
				$3241$ & $(3,2), \textcolor{red}{(3,4)}, (2,1), (4,1)$ & $1$\\
				\hline
				$1432$ & $\textcolor{red}{(1,4)}, (4,3), (3,2)$ & $1$ \\
				\hline 
				$4123$ & $(4,1), \textcolor{red}{(1,2), (2,3)}$ & $2$ \\
				\hline
				$2413$ & $\textcolor{red}{(2,4)}, (4,1), \textcolor{red}{(1,3)}$ & $2$\\
				\hline
				$3142$ & $(3,1), \textcolor{red}{(3,4), (1,2)}, (4,2)$ & $2$ \\
				\hline 
				$3214$ & $ (3,2), (2,1), \textcolor{red}{(1,4)}$ & $1$ \\
				\hline 
				$2341$ & $\textcolor{red}{(2,3), (3,4)}, (4,1)$ & $2$ \\
				\hline
				$1423$ & $\textcolor{red}{(1,4)}, (4,2), \textcolor{red}{(2,3)}$ & $2$ \\
				\hline 
				$1342$ & $\textcolor{red}{(1,3), (3,4)}, (4,2)$ & $2$ \\
				\hline 
				$2143$ & $(2,1), \textcolor{red}{(1,4)}, (4,3)$ & $1$ \\
				\hline 
				$3124$ & $(3,1),  \textcolor{red}{(1,2), (2,4)}$ & $2$ \\
				\hline 
				$2314$ & $\textcolor{red}{(1,4), (2,3)}, (3,1)$ & $2$ \\
				\hline 
				$1243$ & $\textcolor{red}{(1,2), (2,4)}, (4,3)$ & $2$ \\
				\hline 
				$1324$ & $\textcolor{red}{(1,3)}, (3,2), \textcolor{red}{(2,4)}$ & $2$ \\
				\hline 
				$2134$ & $(2,1), \textcolor{red}{(1,3), (3,4)}$ & $2$ \\
				\hline 
				$1234$ & $\textcolor{red}{(1,2), (2,3), (3,4)}$ & $3$ \\
				\hline
			\end{tabular}
			\caption{\normalsize $w = 4231$.}
			\label{table_Ewu_4231}
		\end{subtable}\hspace{1em}~
		\begin{subtable}[b]{0.5\textwidth}
			\small
			\centering
			\begin{tabular}{|c|c|c|}
				\hline
				$u$ & $\Refwu$ & $\awu$ \\
				\hline
				$3412$ & $(3,1), (3,2), (4,1), (4,2)$ & $0$\\
				\hline
				$1432$ & $\textcolor{red}{(1,3)}, (4,3), (3,2)$ & $1$ \\
				\hline
				$2413$ & $ (2,1), \textcolor{red}{(2,3)},  (4,1), (4,3)$ & $1$ \\
				\hline
				$3142$ & $ (3,1), \textcolor{red}{(1,4)}, (4,2)$ & $1$ \\
				\hline
				$3214$ & $ (3,2), (2,1), \textcolor{red}{(2,4)}$ & $1$ \\
				\hline 
				$1423$ & $\textcolor{red}{(1,2)},  (4,2), \textcolor{red}{(2,3)}$ & $2$ \\
				\hline
				$1342$ & $\textcolor{red}{(1,3), (3,4)}, (4,2)$ & $2$  \\ 
				\hline
				$2143$ & $ (2,1), \textcolor{red}{(1,4)},  (4,3)$ & $1$ \\
				\hline
				$3124$ & $ (3,1), \textcolor{red}{(1,2), (2,4)}$ & $2$ \\
				\hline
				$2314$ & $\textcolor{red}{(2,3)}, (3,1), \textcolor{red}{(3,4)}$ &$2$ \\
				\hline
				$1243$ & $\textcolor{red}{(1,2), (2,4)},  (4,3)$ & $2$ \\
				\hline
				$1324$ & $\textcolor{red}{(1,3)},  (3,2), \textcolor{red}{(2,4)}$ & $2$ \\
				\hline
				$2134$ & $ (2,1), \textcolor{red}{(1,3), (3,4)}$ & $2$ \\
				\hline
				$1234$ & $\textcolor{red}{(1,2),(2,3),(3,4)}$ & $3$ \\ \hline
			\end{tabular}
			\caption{\normalsize $w = 3412$.}
			\label{table_Ewu_3412}
		\end{subtable}
		\caption{$\Refwu$ and $\awu$.}
		\label{table_Refwu_and_dwu}
	\end{table}
\end{example}
\begin{remark}
	Looking at descents
	\[ \dwu:= \# \{ (u(i),u(j))\in E_w(u)\mid u(i)>u(j)\} \quad\text{for any $u\le w$}, \]
	one can define a polynomial 
	\[ \bar{A}_w(t) := \sum_{u\le w} t^{\dwu}.\]
	When $Y_w$ is smooth, we have $\bar{A}_w(t)=A_w(t)$ but otherwise they may differ.  Indeed
	\[
	\begin{split}
	\bar{A}_w(t)&=t^4+2t^3+6t^2+10t+1 \quad\text{when $w=4231$},\\
	\bar{A}_w(t)&=t^4+t^3+4t^2+7t+1 \ \quad\text{when $w=3412$}.
	\end{split}
	\]
	
\end{remark}

As mentioned in Remark~\ref{rema:E_u(w)},
\[
\Refwou=\{ (u(i), u(i+1))\mid i=1,2,\dots,n-1\},
\]
so $\awuo$ is the number of ascents in $u$ and hence $A_{w_0}(n,k)$ is the number of permutations in $\mathfrak{S}_n$ with $k$ ascents, which is the Eulerian number $A(n,k)$. Therefore, $A_{w_0}(t)$ is the Eulerian polynomial $A_n(t)$. As is well-known, the Poincar\'e polynomial of the permutohedral variety $\Ywo$ is given by $A_n(t^2)=A_{w_0}(t^2)$. The following theorem generalizes this fact. 

\begin{theorem} \label{theo:Poincare}
	If $\Yw$ is smooth, then its Poincar\'e polynomial agrees with $A_w(t^2)$ and hence $A_w(t)$ is palindromic and unimodal. 
\end{theorem} 
\begin{remark}\label{remark_poincare_polynomial_of_4231_3412}
	When $w=4231$ or $3412$, $Y_w$ is singular and $A_w(t)$ is not palindromic but one can see that $A_w(t^2)$ still agrees with the Poincar\'e polynomial of $Y_w$ (see~\ref{section_retraction_sequence_poincare_polynomial}).  
\end{remark}

\begin{proof}[Proof of Theorem~\ref{theo:Poincare}]

	The maximal cone $\Cbwu$ in the fan of $\Yw$ corresponds to the fixed point $uB$ and its dual cone $\Dwu$ is spanned by $e_{u(j)}-e_{u(i)}$'s for $(u(i),u(j))\in \Refwu$. Suppose that $\Yw$ is smooth. Then $e_{u(j)}-e_{u(i)}$'s for $(u(i),u(j))\in \Refwu$ are primitive inward-pointing edge vectors from the vertex $\mu(u)$ of the moment map image $\mu(Y_w)$. We choose an element $\a=[a_1,\dots,a_n]\in \UZ$ such that $a_1<a_2<\dots<a_n$. 
	Then we have that
	\[
	\langle e_{u(j)}-e_{u(i)},\a\rangle=a_{u(j)}-a_{u(i)} > 0 \iff u(j) > u(i).
	\]	
	Since $\awu$ counts the number of pairs $(u(i),u(j)) \in \Refwu$ such that $u(j) > u(i)$ and the vertices of the moment polytope are given by $\{\mu(u) \mid u \leq w\}$, the Poincar\'e polynomial of $\Yw$ is given by $\sum_{u\le w}t^{2 \awu} =A_w(t^2)$
(see \cite[Theorem~B.I.3.6]{ACL03_symplectic} for more details on computing Poincar\'{e} polynomials of symplectic toric manifolds).

	Since $Y_w$ is smooth and projective, the Poincar\'e duality and the hard Lefschetz theorem imply the palindromicity and unimodality of the Poincar\'e polynomial of~$Y_w$, proving the theorem.   
\end{proof}

Now, since $\Yw$ is smooth if and only if $\Gwu$ is a forest for any $u\le w$ (Corollary~\ref{cor_Duw_smooth_iff_Guw_forest}), we obtain the following corollary from Theorem~\ref{theo:Poincare}. 

\begin{corollary}
	The polynomial $A_w(t)$ is palindromic if the graph $\Gwu$ is a forest for any $u\le w$.
\end{corollary}

We conclude with remarks and questions about the criterion of smoothness of $\Yw$.  
The polytope $\mu(\Yw)$ is the Bruhat interval polytope $Q_{id,w^{-1}}$ in \cite{ts-wi15} for any $w\in \mathfrak{S}_n$, where $id$ denotes the identity permutation, and the vectors $e_{u(j)}-e_{u(i)}$'s for $(u(i),u(j))\in \Refwu$ are the edge vectors of the polytope emanating from the vertex $\mu(uB)$. Note that the $1$-skeleton of the polytope $Q_{id,w^{-1}}$ is the Bruhat graph of $Y_w$ (where the label of the vertex $\mu(uB)$ in the Bruhat graph is $u$ although $\mu(uB)=(u^{-1}(1), \dots, u^{-1}(n))$). Therefore the following are equivalent:
\begin{enumerate}
	\item $\Yw$ is smooth. \label{equi_con_1}
	\item The Bruhat interval polytope $Q_{id,w^{-1}}$ is simple, equivalently the Bruhat graph of $Y_w$ is regular. \label{equi_con_2}
	\item The graph $\Gwu$ is a forest for any $u\le w$. \label{equi_con_3}
\end{enumerate}
Conjecture~\ref{conj_smoothness_of_Yw} says that the following \ref{equi_con_4} would imply \ref{equi_con_3} above:
\begin{enumerate}
	\setcounter{enumi}{3}
	\item The graph $\Gw=\G_w(w)$ is a forest. \label{equi_con_4}
\end{enumerate}
Therefore, if the conjecture is true, then all four statements above are equivalent.  

It is known that the Schubert variety $\Xw$ is smooth if and only if its Poincar\'e polynomial is palindromic (see \cite[Theorem~6.0.4, Corollary~6.1.13, Theorem~6.2.4]{BL20Singular}).  We may ask whether the same holds for $\Yw$, in other words, whether the following~\ref{equi_con_5} is equivalent to \ref{equi_con_1} above:
\begin{enumerate}
	\setcounter{enumi}{4}
	\item The Poincar\'e polynomial of $\Yw$ is palindromic.  \label{equi_con_5}
\end{enumerate}

If $\Yw$ is smooth, then $A_w(t)$ is palindromic (Theorem~\ref{theo:Poincare}) but otherwise $A_w(t)$ may not be palindromic. As in Remark~\ref{remark_poincare_polynomial_of_4231_3412}, $A_w(t^2)$ might agree with the Poincar\'e polynomial of $Y_w$ for any $w\in \mathfrak{S}_n$.  Related to~\ref{equi_con_5}, we may ask whether the following~\ref{equi_con_6} is equivalent to~\ref{equi_con_1} above:
\begin{enumerate}
		\setcounter{enumi}{5}
	\item $A_w(t)$ is palindromic. \label{equi_con_6}
\end{enumerate}
When $n\le 3$, $\Yw$ is smooth for all $w$.  When $n=4$, $\Yw$ is smooth if and only if $w$ is different from $4231$ and $3412$.  As is in Example~\ref{exam:Aw}, $A_w(t)$ is not palindromic when $w=4231$ or $3412$.  Therefore \ref{equi_con_1} and \ref{equi_con_6} above are equivalent for $n\le 4$.

\appendix
\section{Acyclicity of graph and pattern avoidance of permutation}
\label{section_acyclicity_and_avoidance}
In the appendix, we will give a proof of Proposition~\ref{prop:avoidence} which is different from the proof in \cite{bm-bu07}.
If $(a,b)$ is an inversion in $w$, namely $a>b$ and $w^{-1}(a)<w^{-1}(b)$, then there is a sequence of pairs $(a_i,b_i)\in \Refw$ $(i=1,\dots,k)$ such that 
\[
a=a_1> b_1=a_{2}> b_2=a_3> \dots > b_{k-1}=a_k> b_k=b.
\]
This produces a path from $a$ to $b$ in $\Gw$ and we call such a path a \emph{descending path} from $a$ to $b$. 
There is a descending path from $a$ to $b$ if and only if $(a,b)$ is an inversion of $w$. We denote by $\Inv(w)$ the set of inversions of $w$.

\begin{proof}[Proof of Proposition~\ref{prop:avoidence}]
	The proposition is equivalent to the statement that $w$ has a pattern $4231$ or $4512$, which is not a subsequence of $45312$ if and only if $\Gw$ has a cycle, and we shall prove this equivalent statement. 
	
	The only if part $(\Longrightarrow)$. 
	{\bf Case 1}. The case where $w$ has the pattern $4231$ (see Figure~\ref{fig1}). In this case there are $a,b,c,d\in [n]$ such that 
	\begin{equation*} \label{eq:4231}
	w^{-1}(a)<w^{-1}(b)<w^{-1}(c)<w^{-1}(d)\qquad\text{and}\qquad a>c>b>d.
	\end{equation*}
	Since $(a,b), (b,d)\in \Inv(w)$, there is a descending path $P$ in $\Gw$ from $a$ to $d$ via $b$. Note that $P$ does not contain $c$ because $b\in P$ and $(b,c)\notin \Inv(w)$. Similarly, since $(a,c), (c,d)\in \Inv(w)$, there is a descending path $Q$ in $\Gw$ from $a$ to $d$ via $c$ and $Q$ does not contain $b$ because $c\in Q$ and $(b,c)\notin\Inv(w)$. Therefore the union $P\cup Q$ contains a cycle. 
	
	\smallskip
	{\bf Case 2}. The case where $w$ has the pattern $4512$ which is not a subsequence of $45312$ (see~Figure~\ref{fig2}). The argument is similar to Case 1 but rather subtle. In this case there are $a,b,c,d\in [n]$ such that 
	\begin{equation} \label{eq:3412}
	w^{-1}(a)<w^{-1}(b)<w^{-1}(c)<w^{-1}(d)\qquad\text{and}\qquad b>a>d>c
	\end{equation} 
	and we may assume that 
	\begin{equation} \label{eq:badc}
	\begin{split}
	&\text{there is no $p$ such that $a>p>d$ and $w^{-1}(a)<w^{-1}(p)<w^{-1}(d)$, and}\\
	&\text{there is no $q$ such that $b>q>c$ and $w^{-1}(b)<w^{-1}(q)<w^{-1}(c)$}.
	\end{split}
	\end{equation}
	Since $(a,d)\in \Inv(w)$, there is a descending path $P_1$ from $a$ to $d$ and since $(a,c)\in \Inv(w)$, there is a descending path $P_2$ from $a$ to $c$. Both $P_1$ and $P_2$ do not contain the vertex $b$ because $b>a$. Note that 
	\[
	P_1\cap P_2=\{a\},\ \text{so $P_1\cup P_2$ is a path which joins $c$ and $d$ and does not contain $b$}. 
	\]
	
	\begin{figure}[tb]
		\begin{center}
			\begin{subfigure}[t]{0.5\textwidth}
				\centering
				\tikzstyle{vertex} = [circle,draw, fill=white!20, inner sep=0pt, minimum size = 4mm]
				\begin{tikzpicture}[scale = 0.6]
				\draw[->, gray, very thin] (0,0) -- (5,0) node[very near end, sloped, below] {position};
				\draw[->, gray, very thin] (0,0) -- (0,5) node[very near end, sloped, above] {value};
				
				\node (a) [vertex] at (1,4) {$a$};
				\node (b) [vertex] at (2,2) {$b$};
				\node (c) [vertex] at (3,3) {$c$};
				\node (d) [vertex] at (4,1) {$d$};

				\draw[->-=.5, red] (a) -- node[midway, below, sloped] {\small$P$} (b);
				\draw[->-=.5, red] (b) -- (d);
				\draw[->-=.5, blue] (a) -- node[midway, above, sloped] {\small$Q$} (c);
				\draw[->-=.5, blue] (c) -- (d);
				\end{tikzpicture}
				\caption{$w$ has pattern $4231$.}
				\label{fig1}
			\end{subfigure}%
			\begin{subfigure}[t]{0.5\textwidth}
				\centering
				\tikzstyle{vertex} = [circle,draw, fill=white!20, inner sep=0pt, minimum size = 4mm]
				\begin{tikzpicture}[scale = 0.6]
				\draw[->, gray, very thin] (0,0) -- (5,0) node[very near end, sloped, below] {position};
				\draw[->, gray, very thin] (0,0) -- (0,5) node[very near end, sloped, above] {value};
				
				\node (a) [vertex] at (1,3) {$a$};
				\node (b) [vertex] at (2,4) {$b$};
				\node (c) [vertex] at (3,1) {$c$};
				\node (d) [vertex] at (4,2) {$d$};
				
				\draw[->-=.5, red] (a) to node[auto, swap, sloped] {\small$P_2$} (c);
				\draw[->-=.7, blue] (a) -- node[very near end, sloped, below] {\small$P_1$} (d);
				\draw[->-=.5, purple] (b) to node[midway, sloped, above] {\small$Q_2$} (d);
				\draw[->-=.7, orange!70!red] (b) -- node[very near start, sloped, below] {\small$Q_1$} (c);
				
				\end{tikzpicture}
				\caption{$w$ has pattern $4512$ which is not a subsequence of $45312$.}
				\label{fig2}
			\end{subfigure}
			\caption{Graph $\Gw$.}
		\end{center}
	\end{figure}
	Indeed, the vertex $a$ is an end point of $P_1$ and $P_2$ and if $P_1$ and $P_2$ has another common vertex, say $p$, then $a>p>d$ and $w^{-1}(a)<w^{-1}(p)<w^{-1}(c)$ but this contradicts \eqref{eq:badc} because $w^{-1}(c)<w^{-1}(d)$ by \eqref{eq:3412}. 
	Similarly, since $(b,c)\in \Inv(w)$, there is a descending path $Q_1$ from $b$ to $c$, and since $(b,d)\in \Inv(w)$, there is a descending path $Q_2$ from $b$ to $d$. Both $Q_1$ and $Q_2$ do not contain the vertex $a$ because they are descending paths and $w^{-1}(a)<w^{-1}(b)$. Note also that 
	\[
	Q_1\cap Q_2=\{b\},\ \text{so $Q_1\cup Q_2$ is a path which joins $c$ and $d$ and does not contain $a$}. 
	\]
	Indeed, the vertex $b$ is an end point of $Q_1$ and $Q_2$ and if $Q_1$ and $Q_2$ have another common vertex, say $q$, then $b>q>d$ and $w^{-1}(b)<w^{-1}(q)<w^{-1}(c)$ but this contradicts \eqref{eq:badc} because $d>c$ by \eqref{eq:3412}. 
	
	Both $P:=P_1\cup P_2$ and $Q:=Q_1\cup Q_2$ are \emph{paths} joining $c$ and $d$, and $a\in P$ but $a\notin Q$. 
	Therefore, the union $P\cup Q$ contains a cycle. 
	
	\medskip
	The if part $(\Longleftarrow)$. Suppose that $\Gw$ contains a cycle $S$. We may assume that $S$ is a circle. 
	
	\smallskip
	\noindent
	{\bf Claim.} The cycle $S$ contains at least four vertices. 
	
	\smallskip
	
	\noindent
	Indeed, if it has only three vertices, say $x,y,z$ and $w^{-1}(x)<w^{-1}(y)<w^{-1}(z)$, then since $S$ is a triangle, we must have $(x,y), (y,z), (x,z)\in \Refw$ (Figure~\ref{fig3}) and the belonging of the first two elements to $\Refw$ implies $x>y>z$ but this contradicts $(x,z)$ being in $\Refw$. Therefore the claim follows. 
	
	\smallskip
	Let $a$ (resp. $d$) be the vertex of $S$ such that $w^{-1}(a)\le w^{-1}(x)$ (resp. $w^{-1}(x)\le w^{-1}(d)$) for all vertices $x$ of $S$, in other words, 
	\[
	\begin{split}
	&\text{$a$ is leftmost while $d$ is rightmost among vertices of $S$}\\
	&\text{in the one-line notation for $w$.}
	\end{split}
	\]
	We call the interval $[w^{-1}(a),w^{-1}(d)]$ the \emph{width} of $S$. We may assume that the width of $S$ is minimal, i.e., there is no cycle $S'$ in $\Gw$ such that the width of $S'$ is properly contained in the width of $S$. 
	We say that a descending path $P$ from $x$ to $y$ is \emph{emanating} from $x$ or is \emph{terminating} at $y$. Note that since $a$ (resp. $d$) is the leftmost (resp. rightmost) vertex of $S$, the two edges at the vertex $a$ (resp. $d$) in $S$ are emanating from $a$ (resp. terminating at $d$). 
	
	We consider two cases. 
	
	\smallskip
	\noindent
	{\bf Case I}. The case where $a$ and $d$ are joined by an edge, i.e., $(a,d)\in \Refw$. In this case we will find vertices $b$ and $c$ of $S$ such that $abcd$ is the pattern $4512$ which is not a subsequence of $45312$. 
	
	Since $(a,d)\in\Refw$, we have 
	\begin{equation} \label{eq:ad}
	a>d.
	\end{equation}
	
	We note that the path from $a$ to $d$ (outside of the edge $(a,d)$) is not a descending path because if so, any vertex $x$ in the middle of the path (Figure~\ref{fig4}) satisfies $w^{-1}(a)<w^{-1}(x)<w^{-1}(d)$ and $a>x>d$ but this contradicts $(a,d)$ being in $\Refw$. Therefore, 
	there is a vertex $q\ (q\not=a, d)$ of $S$ such that the descending path from $a$ to $q$ is maximal, which means that if $q'$ is the vertex of $S$ next to $q$ but outside of the descending path, then $(q',q)\in \Refw$ (Figure~\ref{fig5q}). We denote this maximal descending path by 
	\begin{equation} \label{eq:Paq}
	P(a,q). 
	\end{equation}
	Similarly, there is a vertex $p\ (\not=a, d,q)$ of $S$ such that the path
	\begin{equation} \label{eq:Ppd}
	\text{$P(p,d)$}
	\end{equation}
	from $p$ to $d$ which does not contain $a$ is a maximal descending path terminating at $d$, which means that if $p'$ is the vertex of $S$ next to $p$ and outside of $P(p,d)$, then $(p,p')\in \Refw$ (Figure~\ref{fig5p}). 
	\begin{figure}[tb]
		\begin{center}
			\begin{minipage}[t]{0.25\textwidth}
				\centering
				\tikzstyle{vertex} = [circle,draw, fill=white!20, inner sep=0pt, minimum size = 4mm]
				\begin{tikzpicture}[scale = 0.5]
				
				\node (x) [vertex] at (1,4) {$x$};
				\node (y) [vertex] at (3,3) {$y$};
				\node (z) [vertex] at (4,1) {$z$};

				\draw[->-=.5] (x) -- (y);
				\draw[->-=.5] (x) -- (z);
				\draw[->-=.5] (y) -- (z);
				\end{tikzpicture}
				\caption{}
				\label{fig3}
			\end{minipage}%
			\begin{minipage}[t]{0.25\textwidth}
				\centering
				\tikzstyle{vertex} = [circle,draw, fill=white!20, inner sep=0pt, minimum size = 4mm]
				\begin{tikzpicture}[scale = 0.5]
				
				\node(a) [vertex] at (1,6) {$a$};
				\node(ax) [circle, fill=black, inner sep = 1pt] at (2,4) {};
				\node(x) [vertex] at (3,3) {$x$};
				\node(xd) [circle, fill=black, inner sep = 1pt] at (4,2) {};
				\node (d) [vertex] at (6,1) {$d$};
				
				\draw[->-=.5] (a)--(ax);
				\draw[->-=.5] (ax)--(x);
				\draw[->-=.5] (x)--(xd);
				\draw[->-=.5] (xd)--(d);
				\draw[->-=.5] (a)--(d);

				\end{tikzpicture}
				\caption{}
				\label{fig4}
			\end{minipage}%
			\begin{minipage}[t]{0.25\textwidth}
				\centering
				\tikzstyle{vertex} = [circle,draw, fill=white!20, inner sep=0pt, minimum size = 4mm]
				\begin{tikzpicture}[scale = 0.5]
				
				\node (a) [vertex] at (1,4) {$a$};
				\node (q') [vertex] at (3,3) {$q'$};
				\node (q) [vertex] at (4,1) {$q$};

				\draw[->-=.5, bend right, red] (a) to node[midway, below, sloped] {\small$P(a,q)$} (q);
				\draw[->-=.5] (q') -- (q);
				\end{tikzpicture}
				\caption{}
				\label{fig5q}
			\end{minipage}%
			\begin{minipage}[t]{0.25\textwidth}
				\centering
				\tikzstyle{vertex} = [circle,draw, fill=white!20, inner sep=0pt, minimum size = 4mm]
				\begin{tikzpicture}[scale = 0.5]
				
				\node (a) [vertex] at (1,3) {$a$};
				\node (p) [vertex] at (2,4) {$p$};
				\node (p') [vertex] at (3,2) {$p'$};
				\node (d) [vertex] at (4,1) {$d$};
				
				\draw[->-=.5, bend right] (a) to (d);
				\draw[->-=.5] (p)--(p');
				\draw[->-=.5, bend left, red] (p) to node[midway, above, sloped] {\small$P(p,d)$} (d);
				
				\end{tikzpicture}
				\caption{}
				\label{fig5p}
			\end{minipage}
		\end{center}
	\end{figure}
	It follows that 
	\begin{equation} \label{eq:pq}
	\begin{split}
	&a>q, \quad p>d\\
	& p>p',\quad w^{-1}(p)<w^{-1}(p'),\\
	&q'>q,\quad w^{-1}(q')<w^{-1}(q). 
	\end{split}
	\end{equation}
	(Although we will not use the last two inequalities above for $q$ and $q'$, the argument below will work if we use $q$ and $q'$ instead of $p$ and $p'$.)

	We claim 
	\begin{equation} \label{eq:padq}
	p>a>d>q.
	\end{equation}
	Indeed, if $a>p$ (Figure~\ref{fig6p}), then $a>p>d$ by \eqref{eq:pq} and $w^{-1}(a)<w^{-1}(p)<w^{-1}(d)$ since $a$ is the leftmost vertex of $S$ and $d$ is the rightmost vertex of $S$. This contradicts $(a,d)$ being in $\Refw$. Therefore $p>a$. Similarly, if $q>d$ (Figure~\ref{fig6q}), then $a>q>d$ by \eqref{eq:pq} and $w^{-1}(a)<w^{-1}(q)<w^{-1}(d)$ but this contradicts $(a,d)$ being in $\Refw$. Therefore $d>q$. These together with \eqref{eq:ad} prove \eqref{eq:padq}. 
	
	\smallskip
	Now we take two cases. 
	
	(1) The case where $w^{-1}(p)<w^{-1}(q)$ (Figure~\ref{fig7-1}). In this case, we take $b=p$ and $c=q$. Then $w^{-1}(a)<w^{-1}(b)<w^{-1}(c)<w^{-1}(d)$ by assumption and $b>a>d>c$ by \eqref{eq:padq}. Therefore $abcd$ is the pattern $4512$. Moreover, since $(a,d)\in \Refw$, there is no $r$ such that $a>r>d$ and $w^{-1}(a)<w^{-1}(r)<w^{-1}(d)$, which means that $abcd$ is not a subsequence of $45312$.
	\begin{figure}[b!]
		\begin{center}
			\begin{minipage}[t]{0.25\textwidth}
				\centering
				\tikzstyle{vertex} = [circle,draw, fill=white!20, inner sep=0pt, minimum size = 4mm]
				\begin{tikzpicture}[scale = 0.5]
				
				\node (x) [vertex] at (1,4) {$a$};
				\node (y) [vertex] at (3,3) {$p$};
				\node (z) [vertex] at (4,1) {$d$};

				\draw[->-=.5] (x) to (y);
				\draw[->-=.5] (x) -- (z);
				\draw[->-=.5] (y) to (z);
				\end{tikzpicture}
				\caption{}
				\label{fig6p}
			\end{minipage}%
			\begin{minipage}[t]{0.25\textwidth}
				\centering
				\tikzstyle{vertex} = [circle,draw, fill=white!20, inner sep=0pt, minimum size = 4mm]
				\begin{tikzpicture}[scale = 0.5]
				
				\node (x) [vertex] at (1,4) {$a$};
				\node (y) [vertex] at (2,2) {$q$};
				\node (z) [vertex] at (4,1) {$d$};

				\draw[->-=.5] (x) to (y);
				\draw[->-=.5] (x) -- (z);
				\draw[->-=.5] (y) to (z);

				\end{tikzpicture}
				\caption{}
				\label{fig6q}
			\end{minipage}%
			\begin{minipage}[t]{0.5\textwidth}
				\centering
				\tikzstyle{vertex} = [circle,draw, fill=white!20, inner sep=0pt, minimum size = 4mm]
				\begin{tikzpicture}[scale = 0.5]
				
				\node (a) [vertex] at (1,5) {$a$};
				\node (q') [vertex] at (3,3) {$q'$};
				\node (q) [vertex] at (4,1) {$q$};
				\node (p) [vertex] at (2,7) {$p$};
				\node (p') [vertex] at (5,4) {$p'$};
				\node (d) [vertex] at (7,2) {$d$};

				\draw[->-=.5, bend right, red] (a) to node[midway, below, sloped] {\small$P(a,q)$} (q);
				\draw[->-=.5] (q') -- (q);
				
				\draw[->-=.4] (a)--(d);
				\draw[->-=.5] (p)--(p');
				\draw[->-=.5, bend left, blue] (p) to node[midway, above, sloped] {\small$P(p,d)$} (d);
				
				\draw[dashed] (p') -- (q');
				
				\end{tikzpicture}
				\caption{}
				\label{fig7-1}
			\end{minipage}
		\end{center}
	\end{figure}
	
	(2) The case where $w^{-1}(q)<w^{-1}(p)$ (Figure~\ref{fig7-2}). We look at $p'$ defined above (a similar argument will work if we loot at $q'$). We claim
	\begin{equation} \label{eq:pqprime}
	d>p'. 
	\end{equation}
	Indeed, if $p'>d$ (like Figure~\ref{fig7-2}), then $(p',d)\in\Inv(w)$ and hence there is a descending path 
	\begin{equation*} \label{eq:Ppprimed}
	P(p',d)
	\end{equation*} 
	from $p'$ to $d$ and $P(p',d)$ does not contain $p$ because $w^{-1}(p)<w^{-1}(p')$. 
	Therefore the union 
	\[
	(p,p')\cup P(p',d)
	\]
	is a descending path from $p$ to $d$. On the other hand, $P(p,d)$ in \eqref{eq:Ppd} is also a descending path from $p$ to $d$ but does not contain $p'$. Therefore the union 
	\[
	\big((p,p')\cup P(p',d)\big)\cup P(p,d)
	\] 
	contains a cycle and its width is $[w^{-1}(p),w^{-1}(d)]$. Since this width is properly contained in the width $[w^{-1}(a),w^{-1}(d)]$ of $S$, this contradicts the minimality of the width of $S$. Therefore $d>p'$.

	
	By \eqref{eq:padq} and \eqref{eq:pqprime} we have 
	\begin{equation} \label{eq:qadpprime}
	p>a>d>p'\qquad \text{(Figure~\ref{fig8})}.
	\end{equation}
	\begin{figure}[tb]
		\begin{center}
			\begin{minipage}[t]{0.5\textwidth}
				\centering
				\tikzstyle{vertex} = [circle,draw, fill=white!20, inner sep=0pt, minimum size = 4mm]
				\begin{tikzpicture}[scale = 0.5]
				
				\node (a) [vertex] at (1,5) {$a$};
				\node (q') [vertex] at (3,3) {$q'$};
				\node (q) [vertex] at (4,1) {$q$};
				\node (p) [vertex] at (5,7) {$p$};
				\node (p') [vertex] at (6,4) {$p'$};
				\node (d) [vertex] at (8,2) {$d$};
				
				\tikzstyle{every pin edge}=[solid, <-,snake=snake,line before snake=5pt]
				
				\draw[->-=.5, bend right, red] (a) to node[midway, below, sloped] {\small$P(a,q)$} (q);
				\draw[->-=.5] (q') -- (q);
				
				\draw[->-=.4] (a)--(d);
				\draw[->-=.5] (p)--(p');
				\draw[->-=.5, bend left, blue] (p) to node[midway, above, sloped] {\small$P(p,d)$} (d);
				
				\draw[dashed] (p') -- (q');
				\draw[dashed, ->-=.7, purple] (a) to node[midway, pin=85:{$P(a,q')$}, purple] {} (q');
				\draw[dashed, ->-=.3, orange!70!red] (p') to node[midway, pin=270:{$P(p',d)$}, orange!70!red] {} (d);
				\end{tikzpicture}
				\caption{}
				\label{fig7-2}
			\end{minipage}%
			\begin{minipage}[t]{0.5\textwidth}
				\centering
				\tikzstyle{vertex} = [circle,draw, fill=white!20, inner sep=0pt, minimum size = 4mm]
				\begin{tikzpicture}[scale = 0.5]
				
				\node (a) [vertex] at (1,4) {$a$};
				\node (q) [vertex] at (3,2) {$q$};
				\node (q') [vertex] at (2,6) {$q'$};
				\node (p) [vertex] at (5,5) {$p$};
				\node (p') [vertex] at (6,1) {$p'$};
				\node (d) [vertex] at (7,3) {$d$};
				
				\draw[->-=.5] (a)--(q);
				\draw[->-=.5] (q')--(q);
				\draw[->-=.6] (p)--(p');
				\draw[->-=.5] (p)--(d);
				\draw[->-=.4] (a)--(d);
				\draw[dashed] (q')--(p');
				
				\end{tikzpicture}
				\caption{}
				\label{fig8}
			\end{minipage}
		\end{center}
	\end{figure}
	
	Moreover, since $a$ is the leftmost vertex of the cycle $S$ and $d$ is the rightmost vertex of $S$, it follows from \eqref{eq:pq} that 
	\[
	w^{-1}(a)<w^{-1}(p)<w^{-1}(p')<w^{-1}(d).
	\]
	This together with \eqref{eq:qadpprime} shows that if we take $b=p$ and $c=p'$, then $abcd$ is the pattern $4512$. Moreover, since $(a,d)\in \Refw$, there is no $r$ such that $a>r>d$ and $w^{-1}(a)<w^{-1}(r)<w^{-1}(d)$, which means that $abcd$ is not a subsequence of $45312$.
	
	\medskip
	\noindent
	{\bf Case II.} The case where $a$ and $d$ are not joined by an edge, i.e., $(a,d)\notin \Refw$. In this case we will find vertices $b$ and $c$ of $S$ such that $abcd$ is the pattern $4231$. There are two edges of $S$ emanating from $a$ and we take vertices $p,q\in S$ such that $(a,p), (a,q)\in \Refw$. We may assume 
	\begin{equation} \label{eq:wpq}
	w^{-1}(p)<w^{-1}(q)
	\end{equation}
	without loss of generality (Figures~\ref{fig9-1} and~\ref{fig9-2}). Since $(a,p), (a,q)\in \Refw$, we have 
	\begin{equation*} 
	a>p,\qquad a>q.
	\end{equation*}
	We claim $q>p$ so that 
	\begin{equation} \label{eq:aqp}
	a>q>p\qquad(\text{i.e., Figure~\ref{fig9-2} does not occur}). 
	\end{equation}
	Indeed, if $p>q$ (Figure~\ref{fig9-2}), then $a>p>q$ and $w^{-1}(a)<w^{-1}(p)<w^{-1}(q)$ by~\eqref{eq:wpq}. But this contradicts $(a,q)$ being in $\Refw$. 
	
	\smallskip
	Now we take two cases. 
	
	(1) The case where $p>d$ (Figure~\ref{fig10}). 	
	\begin{figure}[b]
		\begin{center}
			\begin{minipage}[t]{0.33\textwidth}
				\centering
				\tikzstyle{vertex} = [circle,draw, fill=white!20, inner sep=0pt, minimum size = 4mm]
				\begin{tikzpicture}[scale = 0.5]
				
				\node (a) [vertex] at (1,4) {$a$};
				\node (p) [vertex] at (2,2) {$p$};
				\node (q) [vertex] at (4,3) {$q$};
				
				\draw[->-=.5] (a) -- (p);
				\draw[->-=.5] (a) -- (q);
				
				\end{tikzpicture}
				\caption{}
				\label{fig9-1}
			\end{minipage}%
			\begin{minipage}[t]{0.33\textwidth}
				\centering
				\tikzstyle{vertex} = [circle,draw, fill=white!20, inner sep=0pt, minimum size = 4mm]
				\begin{tikzpicture}[scale = 0.5]
				
				\node (a) [vertex] at (1,4) {$a$};
				\node (p) [vertex] at (3,3) {$p$};
				\node (q) [vertex] at (4,1) {$q$};
				
				\draw[->- =.5] (a)--(p);
				\draw[->- =.5] (a)--(q);
				
				\end{tikzpicture}
				\caption{}
				\label{fig9-2}
			\end{minipage}%
			\begin{minipage}[t]{0.33\textwidth}
				\centering
				\tikzstyle{vertex} = [circle,draw, fill=white!20, inner sep=0pt, minimum size = 4mm]
				\begin{tikzpicture}[scale = 0.5]
				
				\node (a) [vertex] at (1,4) {$a$};
				\node (p) [vertex] at (2,2) {$p$};
				\node (q) [vertex] at (4,3) {$q$};
				\node (d) [vertex] at (5,1) {$d$};
				
				\draw[->-=.5] (a) -- (p);
				\draw[->-=.5] (a) -- (q);

				\end{tikzpicture}
				\caption{}
				\label{fig10}
			\end{minipage}
		\end{center}
	\end{figure}
	In this case $a>q>p>d$ by \eqref{eq:aqp} and $w^{-1}(a)<w^{-1}(p)<w^{-1}(q)<w^{-1}(d)$ by \eqref{eq:wpq}. Therefore if we take $b=p$ and $c=q$, then $abcd$ is the pattern $4231$.

	(2) The case where $d>p$. We shall observe that this case does not occur. Let $p'$ be the vertex of $S$ such that the path 
	\begin{equation*} \label{eq:Papprime}
	P(p,p')
	\end{equation*} 
	in $S$ joining $p$ and $p'$ is a maximal descending path from $p$ to $p'$ so that if $p''$ is the vertex next to $p'$ and outside of $P(p,p')$, then $(p'',p')\in \Refw$ (Figures~\ref{fig11-1} and~\ref{fig11-2}). By the choice of $p'$ and $p''$ we have 
	\begin{equation} \label{eq:pp}
	w^{-1}(p)\leq w^{-1}(p'),\quad p\geq p'\qquad \text{and}\qquad w^{-1}(p'')<w^{-1}(p'),\quad p''>p'.
	\end{equation}
	
	If $w^{-1}(q)<w^{-1}(p')$ (Figure~\ref{fig11-1}), then there is a descending path $P(q,p')$ from $q$ to $p'$ in $\Gw$ because $q>p'$ by \eqref{eq:aqp} and \eqref{eq:pp}. Since $P(q,p')$ does not contain $p$ by \eqref{eq:wpq}, the union $(a,q)\cup P(q,p')$ is a descending path from $a$ to $p'$ which does not contain $p$. On the other hand, the union $(a,p)\cup P(p,p')$ is also a descending path from $a$ to $p'$ but contains $p$. Therefore the union 
	\[
	\big((a,q)\cup P(q,p')\big)\cup \big((a,p)\cup P(p,p')\big)
	\] 
	contains a cycle and its width is $[w^{-1}(a),w^{-1}(p')]$. But this contradicts the minimality of the width of $S$. Therefore 
	\begin{equation} \label{eq:wpprimeq}
	w^{-1}(p')<w^{-1}(q)\qquad \text{(Figure~\ref{fig11-2})}.
	\end{equation}
	
	By \eqref{eq:pp} and \eqref{eq:wpprimeq}, we have
	\[
	w^{-1}(p'')<w^{-1}(q).
	\]
	If $a>p''$ (like Figure~\ref{fig11-2}), then $(a,p'')\in\Inv(w)$ and hence there is a descending path $P(a,p'')$ from $a$ to $p''$ in $\Gw$. The union $P(a,p'')\cup (p'',p')$ is a descending path from $a$ to $p'$. On the other hand, the union $(a,p)\cup P(p,p')$ is also a descending path from $a$ to $p'$ but does not contain $p''$ because $p''$ is the vertex of $S$ next to $p'$ but outside of $P(p,p')$. Therefore the union 
	$$\big(P(a,p'')\cup (p'',p')\big)\cup \big((a,p)\cup P(p,p')\big)$$
	contains a cycle and its width is $[w^{-1}(a),w^{-1}(p')]$. But this contradicts the minimality of the width of $S$. Therefore $p''>a$ and hence 
	\[
	p''>a>q\qquad\text{(Figure~\ref{fig12})}
	\]
	by \eqref{eq:aqp}. 
	\begin{figure}
		\begin{center}
			\begin{minipage}[t]{0.33\textwidth}
				\centering
				\tikzstyle{vertex} = [circle,draw, fill=white!20, inner sep=0pt, minimum size = 4mm]
				\begin{tikzpicture}[scale = 0.5]
				
				\node (a) [vertex] at (1,6) {$a$};
				\node (p) [vertex] at (2,4) {$p$};
				\node (q) [vertex] at (4,5) {$q$};
				\node (p'') [vertex] at (6,3) {$p''$};
				\node (p') [vertex] at (7,1) {$p'$};
				
				\draw[->-=.5] (a) -- (p);
				\draw[->-=.5] (a) -- (q);
				\draw[->-=.5] (p'') -- (p');
				\draw[->-=.5, bend right, red] (p) to node[midway, below, sloped] {\small$P(p,p')$} (p');				
				\draw[->-=.5, bend right, dashed, blue] (q) to node[midway, below, sloped] {\small$P(q,p')$} (p');				
				\end{tikzpicture}
				\caption{}
				\label{fig11-1}
			\end{minipage}%
			\begin{minipage}[t]{0.33\textwidth}
				\centering
				\tikzstyle{vertex} = [circle,draw, fill=white!20, inner sep=0pt, minimum size = 4mm]
				\begin{tikzpicture}[scale = 0.5]
				
				\node (a) [vertex] at (1,6) {$a$};
				\node (p) [vertex] at (2,4) {$p$};
				\node (q) [vertex] at (7,5) {$q$};
				\node (p'') [vertex] at (5,3) {$p''$};
				\node (p') [vertex] at (6,1) {$p'$};
				
				\draw[->-=.5] (a) -- (p);
				\draw[->-=.5] (a) -- (q);
				\draw[->-=.5] (p'') -- (p');
				\draw[->-=.5, bend right=15, red] (p) to node[midway, below, sloped] {\small$P(p,p')$} (p');				
				\draw[->-=.5, bend left=20, dashed, blue] (a) to node[midway, below, sloped] {\small$P(a,p'')$} (p'');
				
				\end{tikzpicture}
				\caption{}
				\label{fig11-2}
			\end{minipage}%
			\begin{minipage}[t]{0.33\textwidth}
				\centering
				\tikzstyle{vertex} = [circle,draw, fill=white!20, inner sep=0pt, minimum size = 4mm]
				\begin{tikzpicture}[scale = 0.5]
				
				\node (a) [vertex] at (1,6) {$a$};
				\node (p) [vertex] at (2,4) {$p$};
				\node (q) [vertex] at (6,5) {$q$};
				\node (p'') [vertex] at (4,7) {$p''$};
				\node (p') [vertex] at (5,1) {$p'$};
				
				\draw[->-=.5] (a) -- (p);
				\draw[->-=.5] (a) -- (q);
				\draw[->-=.5] (p'') -- (p');
				\draw[->-=.5, bend right=15, red] (p) to node[midway, below, sloped] {\small$P(p,p')$} (p');				
				\draw[->-=.5, bend left=20, dashed, blue] (p'') to node[midway, above, sloped] {\small$P(p'',q)$} (q);		
				
				\end{tikzpicture}
				\caption{}
				\label{fig12}
			\end{minipage}
		\end{center}
	\end{figure}
	Then since $w^{-1}(p'')<w^{-1}(q)$ by \eqref{eq:pp} and \eqref{eq:wpprimeq}, we have $(p'',q)\in\Inv(w)$ and hence there is a descending path $P(p'',q)$ from $p''$ to $q$ in $\Gw$. The union 
	\begin{equation} \label{eq:Ppqpp}
	P(p'',q)\cup (p'',p')
	\end{equation} 
	is a \emph{path} joining $p'$ and $q$ (note that $(p'',p')$ is an edge). Similarly, the union 
	\begin{equation} \label{eq:Papaq}
	(a,p)\cup P(p,p')\cup (a,q)
	\end{equation} 
	is also a \emph{path} joining $p'$ and $q$. However, the path in \eqref{eq:Ppqpp} does not contain the vertex $a$ while the path in \eqref{eq:Papaq} does. Therefore the union of these two paths contains a cycle and its width is $[w^{-1}(a),w^{-1}(q)]$. This again contradicts the minimality of the width of $S$. Thus Case II (2) does not occur. 
\end{proof}

\section{Retraction sequences of a polytope and the Poincar\'{e} polynomial}
\label{section_retraction_sequence_poincare_polynomial}
In this appendix, we compute Poincar\'{e} polynomials of $Y_w$ when $w = 4231$ and $w = 3412$. 
We recall the definition of \textit{retraction sequences} of a polytope from \cite{BNSS, BSS}.	

\begin{definition}
	A \defi{polytopal complex} $\mathcal{C}$ is a finite collection of polytopes in $\mathbb{R}^n$ satisfying:
	\begin{enumerate}
		\item if $E$ is a face of $F$ and $F \in \mathcal{C}$ then $E \in \mathcal{C}$.
		\item If $E, F \in \mathcal{C}$ then $E \cap F$ is a face of both $E$ and $F$.
	\end{enumerate}
\end{definition}	

We denote the underlying set of a polytopal complex $\mathcal{C}$ by $|\mathcal{C}| = \bigcup_{F \in \mathcal C} F$. 
Given an $n$-dimensional polytope $Q$, the polytopal complex $\mathcal{C}(Q)$ is a collection of all faces of $Q$.
For a polytopal complex $\mathcal{C}$, we call a vertex $v \in |\mathcal{C}|$ is a \textit{free vertex} if it has a neighborhood in $|\mathcal{C}|$ that is diffeomorphic to $\mathbb{R}^N_{\geq 0}$ as manifolds with corners for some integer $N$. 

\begin{definition}[{\cite[\S 2]{BSS}}]
	A \defi{retraction sequence} of a polytope $Q$ is a sequence of triples $\{(B_k, E_k, b_k)\}_{1 \leq k \leq \ell}$ defined inductively:
	\begin{itemize}
		\item $B_1 = Q = E_1$ and $b_1$ is a free vertex in $Q$, i.e., a simple vertex.
		\item Given $(B_k, E_k, b_k)$, the next term $(B_{k+1}, E_{k+1}, b_{k+1})$ is defined to be:
		\[
		B_{k+1} := |\{ E \in \mathcal{C}(B_k) \mid b_k \notin V(E)\}|,
		\] 
		$b_{k+1}$ is a free vertex in $B_{k+1}$, and $E_{k+1}$ is the maximal face of $B_{k+1}$ containing $b_{k+1}$. 
	\end{itemize}
	Here, $\ell$ is the number of vertices of $Q$.
\end{definition}

It is known that every simple polytope admits a retraction sequence  (see \cite[Proposition 2.3]{BSS}). 	Suppose that $X$ is a projective toric variety whose underlying polytope is $Q$.
If $Q$ admits a retraction sequence, then we can construct a \textbf{q}-CW complex structure on $X$ using the similar argument in the proof of \cite[Proposition~4.4]{BNSS} (This observation is deeply studied in~\cite{SS}). We note that \textbf{q}-CW complex structure is a generalization of CW complex structure using the quotient of a disc by the action of a finite group rather than ordinary cells (see~\cite{PoddarSarkar10}). 
Since a retraction sequence $\{(B_k,E_k,b_k)\}_{1 \leq k \leq \ell}$ produces only even dimensional \lq\lq cells\rq\rq~of  dimensions $\{2 \dim(E_k)\}_{1\leq k \leq \ell}$, one can show the following:
\begin{proposition}\label{prop_retraction_and_Poincare_polynomial}
	Suppose that $X$ is a projective toric variety whose underlying polytope is $Q$.
	If $Q$ admits a retraction sequence $\{(B_k, E_k, b_k )\}_{k=1}^{\ell}$ then the Poincar\'{e} polynomial 
	of $X$ agrees with
	\[
	\sum_{k=1}^{\ell} t^{2\dim(E_k)}. 
	\]
\end{proposition}

When $w = 4231$, the Bruhat interval polytope $Q_{id, w^{-1}}$ is given in Figure~\ref{figure_Bruhat_polytope_4231}. The vertices are labeled by the corresponding permutations, i.e., $\mu(uB)=(u^{-1}(1),\dots,u^{-1}(n))$ is labeled by $u$ for $u \leq w$. 
The polytope $Q_{id, w^{-1}}$ is not simple but it admits a retraction sequence as in~Figure~\ref{figure_retraction_seq_4231}. By Proposition~\ref{prop_retraction_and_Poincare_polynomial}, the Poincar\'{e} polynomial of $Y_w$ is $t^6 + 11 t^4 + 7 t^2 + 1$, which agrees with ${A}_w(t^2)$ (see Table~\ref{table_retraction_seq_4231}).
\begin{figure}
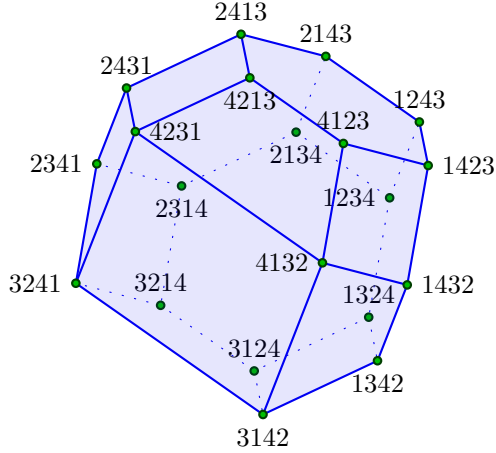

	\begin{center}
%
		\caption{The Bruhat interval polytope $Q_{id, {4231}^{-1}}=Q_{id, 4231}$.}
		\label{figure_Bruhat_polytope_4231}
	\end{center}
\end{figure}

\begin{center}
	\begin{figure}
		\centering
		\def\arraystretch{2}%
		%
		\caption{A retraction sequence $\{(B_k, E_k, b_k)\}$ of $Q_{id, 4231}$.}
		\label{figure_retraction_seq_4231}
	\end{figure}
\end{center}

\begin{table}[H]
	\centering
	\begin{tabular}{|c|cccccccccc|}
		\hline
		$b_k$ & 1234 & 2134 & 1324   & 3124 & 2314 & 1243 & 1423 & 4123 & 2143 & 2413 \\
		\hline
		$\dim(E_k)$ & 3 & 2 & 2 & 2 & 2 & 2 & 2 & 2 & 1 & 2 \\
		\hline
		\hline
		$b_k$ & 1342 & 2341 & 3142 & 1432 & 4132 & 3214 & 3241 & 4213 & 2431 & 4231\\
		\hline
		$\dim(E_k)$ & 2 & 2 & 2 & 1 & 1 & 1 & 1 & 1 & 1 & 0\\
		\hline
	\end{tabular}
	\caption{A sequence of vertices $b_k$ and dimensions $\dim(E_k)$.}
	\label{table_retraction_seq_4231}
\end{table} 	

The Bruhat interval polytope $Q_{id, w^{-1}}$ for $w = 3412$ is given in Figure~\ref{figure_Bruhat_polytope_3412}. The vertex $\mu(uB)=(u^{-1}(1),\dots,u^{-1}(n))$ is labeled by $u$ for $u \leq w$.  
The polytope $Q_{id, w^{-1}}$ is not simple but it admits a retraction sequence as in~Figure~\ref{figure_retraction_seq_3412}. By Proposition~\ref{prop_retraction_and_Poincare_polynomial}, the Poincar\'{e} polynomial of $Y_w$ is $t^6 + 7 t^4 + 5t^2 + 1$, which agrees with ${A}_w(t^2)$ (see Table~\ref{table_retraction_seq_3412}).

\begin{center}
	\begin{figure}[H]
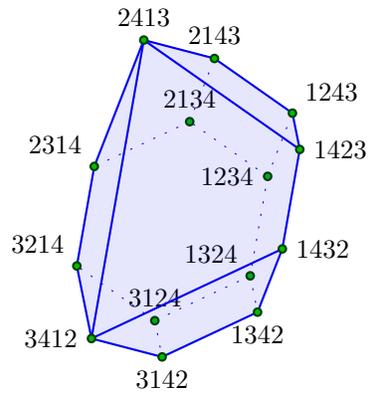

		\centering

		\caption{The Bruhat interval polytope $Q_{id,3412^{-1}}=Q_{id, 3412}$.}
		\label{figure_Bruhat_polytope_3412}
	\end{figure}
\end{center}	

\begin{center}
	\begin{figure}[H]
		\centering
		\def\arraystretch{2}%
		%
		\caption{A retraction sequence $\{(B_k, E_k, b_k)\}$ of $Q_{id, 3412}$.}
		\label{figure_retraction_seq_3412}
	\end{figure}
\end{center}

\begin{center}
	\begin{table}[H]
		\centering
		\begin{tabular}{|c|ccccccc|}
			\hline
			$b_k$ & 1234 & 2134 & 1324 & 3124 & 1243 & 1423 & 1342 \\
			\hline
			$\dim(E_k)$ & 3 &  2 & 2 & 2 & 2 & 2 & 2\\
			\hline 
			\hline 
			$b_k$ & 2314 & 2143 & 2413 & 1432  & 3142 & 3214 & 3412 \\
			\hline
			$\dim(E_k)$ & 2 & 1 & 1 & 1 & 1 & 1 & 0\\
			\hline
		\end{tabular}
		\caption{A sequence of vertices $b_k$ and dimensions $\dim(E_k)$.}
		\label{table_retraction_seq_3412}
	\end{table}
\end{center}












\end{document}